\documentclass[onefignum,onetabnum]{siamart171218}

\usepackage{microtype}
\usepackage{graphicx}
\usepackage{subfigure}
\usepackage{booktabs} 
\usepackage{bm}
\usepackage{xcolor}
\usepackage{algorithm}
\usepackage{algorithmic}

\usepackage[utf8]{inputenc}
\usepackage[english]{babel}

\def\A{{\mathbf A}}

\def\B{{\mathbf B}}
\def\cbf{{\mathbf c}}

\def\e{{\mathbf e}}

\def\I{{\mathbf I}}
\def\L{{\mathbf L}}

\def\M{{\mathbf M}}

\def\p{{\mathbf p}}

\def\q{{\mathbf q}}

\def\s{{\mathbf s}}

\def\u{{\mathbf u}}

\def\v{{\mathbf v}}

\def\w{{\mathbf w}}

\def\x{{\mathbf x}}

\def\y{{\mathbf y}}

\def\z{{\mathbf z}}

\usepackage{lipsum}
\usepackage{amsfonts}
\usepackage{graphicx}
\usepackage{epstopdf}
\usepackage{algorithmic}
\ifpdf
  \DeclareGraphicsExtensions{.eps,.pdf,.png,.jpg}
\else
  \DeclareGraphicsExtensions{.eps}
\fi

\newsiamremark{remark}{Remark}
\newsiamremark{hypothesis}{Hypothesis}
\crefname{hypothesis}{Hypothesis}{Hypotheses}
\newsiamthm{claim}{Claim}

\headers{Bregman Monotone Operator Splitting}{K. Niwa and W. B. Kleijn}

\title{Bregman Monotone Operator Splitting}

\author{Kenta Niwa 
\thanks{
NTT Media Intelligence Laboratories, Japan and Victoria University of Wellington, New Zealand.
(\email{niwa.kenta@lab.ntt.co.jp}). }
\and W. Bastiaan Kleijn 
\thanks{
Victoria University of Wellington, New Zealand
(\email{bastiaan.kleijn@ecs.vuw.ac.nz}).}
}

\usepackage{amsopn}

\ifpdf
\hypersetup{
  pdftitle={Bregman monotone operator splitting},
  pdfauthor={K. Niwa and Bastiaan W. Kleijn}
}
\fi

\begin{document}

\maketitle

\begin{abstract}
Monotone operator splitting is a powerful paradigm that facilitates parallel processing for optimization problems where the cost function can be split into two convex functions. We propose a generalized form of monotone operator splitting based on Bregman divergence. We show that an appropriate design of the Bregman divergence leads to faster convergence than conventional splitting algorithms. The proposed Bregman monotone operator splitting (B-MOS) is applied to an application to illustrate its effectiveness. B-MOS was found to significantly improve the convergence rate. 

\end{abstract}



\section{Introduction}\label{sec:introduction}

Mathematical optimization is commonly used in a wide range of applications including image classification, speech recognition, and natural language processing. In recent years, the performance of optimization algorithms has improved drastically through the use of big data and large computing resources. Although we cannot provide a detailed overview of optimization theory because of the breadth of the field, we can provide a rough categorization on the basis of the three perspectives used in \cite{Liu2017Recent}.

The first perspective is that of the problem formulation. Different formulations of the problem are often possible and lead to different solution methods. Optimization problems are commonly formulated as a cost minimization subject to a set of constraints. Linearly constrained minimization forms are particularly ubiquitous. Moreover, a \textit{dual} formulation \cite{fenchel1949conjugate} forms an alternative that makes the optimization problem more tractable. A famous example dual formulation is the Lagrangian dual ascent problem \cite{uzawa1958gradient}, which is described in Sec. \ref{sec:2_1}. For certain optimization problems, the cost function can be formulated as a summation of components. This is the problem formulation used in \textit{monotone operator splitting (MOS)} e.g., \cite{ryu2016primer, bauschke2017convex}.

The second perspective is that of the solver method. In constructing a solver that applies to the defined problem, its \textit{convergence rate} is an important factor. If we use a deterministic solution method, possible approaches are first order gradient descent (GD) \cite{cauchy1847methode}, accelerated gradient descent (AGD) \cite{nesterov1983method}, and the Newton method and quasi-Newton method \cite{shanno1970conditioning}. To handle large amounts of data, mini-batch based \textit{stochastic optimization} was introduced \cite{robbins1951stochastic}. A method to estimate the convergence rate of stochastic optimization was provided in \cite{bousquet2008tradeoffs}. Recent algorithms, e.g., \cite{johnson2013accelerating, xiao2014proximal, roux2012stochastic, defazio2014saga, suzuki2014stochastic}, are often combinations of a particular problem form and a particular solver. For example, stochastic dual coordinate ascent (SDCA) \cite{shalev2013stochastic} applies first order stochastic gradient descent (SGD) to a risk minimization problem \cite{vapnik1998statistical}, whereas stochastic dual Newton ascent (SDNA) \cite{qu2016sdna} applies second order SGD. If the cost function admits a formulation as a sum of two suitable convex functions, then this naturally leads to MOS approaches, e.g., \cite{ryu2016primer, bauschke2017convex}. The alternating direction method of multipliers (ADMM) \cite{gabay1976dual} is an example of MOS that applies Douglas-Rachford splitting \cite{douglas1956numerical} to the Lagrangian dual ascent problem.

The third perspective is how to run the solver effectively over multiple processors. Recent progress on parallel computing architectures in the context of cloud computing and graphics processing unit (GPU) clusters has resulted in a large research effort towards running solvers in parallel on many processing nodes, usually for big data. Numerous parallel computing methods have been proposed. Pioneering methods include parallelized SGD \cite{zinkevich2010parallelized, zhang2012communication}, the hogwild! algorithm \cite{recht2011hogwild}, elastic averaging SGD \cite{zhang2015deep}, and communication-efficient coordinate ascent (COCOA) \cite{jaggi2014communication}, which is a parallelization form of SDCA. The class of MOS based methods is particularly attractive for running over multiple processors as MOS naturally facilitates parallel computation. Many parallel algorithms based on MOS are variants of ADMM \cite{wei2012distributed, zhang2014asynchronous}. Although ADMM is effective, its convergence rate is often relatively slow because it is based on Douglas-Rachford splitting. The primal-dual method of multipliers (PDMM) \cite{zhang2017distributed, sherson2017derivation} inherently converges faster as it is based on Peaceman-Rachford splitting \cite{peaceman1955numerical}.

In this paper, we focus on MOS methods. Several MOS solvers are well known, such as Peaceman-Rachford splitting \cite{peaceman1955numerical}, Douglas-Rachford splitting \cite{douglas1956numerical}, forward-backward splitting \cite{passty1979ergodic} and Davis-Yin three-operator splitting \cite{davis2017three}. Their variable update procedures are basically composed of operators such as the resolvent and Cayley operators as summarized in \cite{ryu2016primer, bauschke2017convex}. The convergence rates of MOS solvers generally follow from the contractive property of the aforementioned operators. Penalty terms based on the squared $L_{2}$ norm are often used in the variable update cost. This fact indicates that the variables are updated with a pre-determined step-size in a \textit{Euclidean metric} and its convergence rate corresponds to that of first order GD. This suggests that we may obtain a faster convergence rate with MOS based solvers.

Our contribution is a generalization of MOS solvers to using Bregman divergence \cite{bregman1967relaxation} for obtaining faster convergence rates. For conventional MOS methods, the idea of generalization using Bregman divergence has been used in the past. For example, the resolvent operator was generalized to using the Bregman divergence \cite{bauschke2003bregman} and forward-backward splitting was generalized Bregman divergence \cite{van2017forward}. However, this method differs from our method in that the cost function is modified/limited to using the Bregman divergence. In our generalized MOS solvers, such cost modification/limitation is not required and other MOS algorithms, such as Peaceman-Rachford splitting and Douglas-Rachford splitting, are applicable. To make the convergence rate fast, the Bregman divergence must be designed appropriately. As has been discussed for the Newton method, convergence rate improvements are basically due to including higher order convexity, such as the second order gradient (Hessian) in a convex cost \cite{luenberger1973introduction, battiti1992first}. 
Therefore, it is important to investigate the properties of the metric that relate to
convergence rate. We provide a design method for the Bregman divergence metric that leads to fast convergence of
MOS solvers. By means of a convergence rate analysis and numerical experiments, we show
how the convergence rates are affected by the design of the Bregman divergence.

In Sec. \ref{sec:2}, the theory of the new Bregman MOS (B-MOS) is described in the context of deterministic optimization. After Bregman MOS algorithms are constructed, their convergence rates are  predicted. We explain how to design Bregman divergence to achieve fast convergence. In Sec. \ref{sec:3}, the B-MOS solver is applied to a constrained minimization problem to illustrate the effectiveness of B-MOS as an example.

\section{THEORY}
\label{sec:2}

In this section, we generalize the conventional monotone operator splitting algorithms to
use Bregman divergence. The resulting formulation has additional degrees of freedom that can
be optimized to minimize bounds on convergence rate.
After defining the problem, we derive the basic algorithms in Sec. \ref{sec:2_2}. 
In Sec. \ref{sec:2_3}, the fast convergence condition for the Bregman divergence metric is investigated, and 
an implementation of Bregman divergence is explained in Sec. \ref{sec:2_4}.

\subsection{Problem Definition}
\label{sec:2_1}

We consider the problem of finding an infimum of a convex closed proper (CCP) function 
that can be split into two CCP functions as $G(\w) = G_{1}(\w) + G_{2}(\w)$: 
\begin{align}
\inf_{\w}  G_{1}(\w) + G_{2}(\w),
\label{eq:mintwoccp}
\end{align}
where $\w \in \mathbb{R}^{m}$ is the latent variable to be optimized and $G_{i}: \mathbb{R}^{m} \rightarrow \mathbb{R} \cup \{ \infty \}, \hspace{3pt} (i \hspace{-2pt} = \hspace{-2pt} 1, 2)$. 
A fixed point $\w^{\ast}$ can be found by requiring that the subdifferential of
(\ref{eq:mintwoccp}) includes the zero vector,
\begin{align}
\mathbf{0} \in \partial G_{1}(\w^{\ast}) + \partial G_{2}(\w^{\ast}),
\label{eq:subdif_zero}
\end{align}
where $\partial$ is the subdifferential operator \cite{tyrrell1970convex}, and 
$\in$ reflects that its output can be multi-valued. 
As each $G_i (\w)$ is a CCP function, $\partial G_{i} (\w)$ is maximally monotone \cite{minty1960monotone}. 

To show that various relevant problems are of the form (\ref{eq:mintwoccp}), we provide two examples. 

\vspace{5pt}

\noindent
\textbf{Ex. 1: Constrained minimization problem}

\vspace{3pt}

Let us suppose the constrained minimization problem is composed of a CCP loss term
$H_{1}: \mathbb{R}^{m} \rightarrow \mathbb{R} \cup \{ \infty \}$ 
and a CCP regularization term $H_{2}: \mathbb{R}^{m} \rightarrow \mathbb{R} \cup \{ \infty \}$ as
\begin{align}
\inf_{\p} \hspace{2pt} H_{1}(\p) \hspace{10pt} \mathrm{s.t.} \hspace{3pt} H_{2} (\p) \leq 0, 
\label{eq:problem_ex1}
\end{align}
where $\p \in \mathbb{R}^{m}$ is the latent variable. 
It is well known that solving the problem \eqref{eq:problem_ex1} is 
equivalent to solving the following problem \cite{tibshirani1996regression}:
\begin{align}
\inf_{\p}  H_{1}(\p) + \mu H_{2} (\p), 
\label{eq:problem_ex1_reform}
\end{align}
where $\mu > 0$. 
By replacing $\p \rightarrow  \w$, and setting $G_{1}(\w) = H_{1}(\w), G_{2}(\w) = \mu H_{2} (\w)$, 
the constrained minimization problem \eqref{eq:problem_ex1_reform} is of the form \eqref{eq:mintwoccp}. 

\vspace{5pt}

\noindent
\textbf{Ex. 2: Lagrangian dual ascent problem}

\vspace{3pt}

For another constrained minimization problem, 
let us suppose the cost function to be minimized is composed of two CCP functions, 
$H_{1}: \mathbb{R}^{p} \rightarrow \mathbb{R} \cup \{ \infty \}$ and 
$H_{2}: \mathbb{R}^{q} \rightarrow \mathbb{R} \cup \{ \infty \}$, and let the variables be linearly constrained as
\begin{align}
\inf_{\p, \q}  H_{1}(\p) + H_{2}(\q) \hspace{10pt} \rm{s.t.} \hspace{3pt} \A \p + \B \q = \cbf, 
\label{eq:problem_ex2}
\end{align}
where the variables are $\p \in \mathbb{R}^{p}$, $\q  \in \mathbb{R}^{q}$ 
and where $\A \in \mathbb{R}^{m \times p}, \B
\in \mathbb{R}^{m \times q}, \cbf \in \mathbb{R}^{m}$ specify the linear constraint parameters that relate the variables. 
Given the Lagrangian function $\mathcal{L}(\p, \q, \bm{\lambda})$
\begin{align}
\mathcal{L}(\p, \q, \bm{\lambda}) = H_{1}(\p) + H_{2}(\q) + \left< \bm{\lambda}, \cbf - \A \p - \B \q \right>. 
\label{eq:problem_ex2_lagrangian}
\end{align}
When the dual problem exists \cite{fenchel1949conjugate}, 
solving it instead of the primal problem 
is a natural strategy. 
The dual problem takes the form:
\begin{subequations}
\begin{align}
\sup_{\bm{\lambda}} \inf_{\p, \q} \mathcal{L}(\p, \q, \bm{\lambda})
& = \sup_{\bm{\lambda}} \left( - H^{\star}_{1}(\A^{\textrm{T}} \bm{\lambda} ) 
- H^{\star}_{2}(\B^{\textrm{T}} \bm{\lambda} )
+ \left< \bm{\lambda}, \cbf \right>   \right) 
\label{eq:Ldualproblem_ascent}
\\
& = - \inf_{\bm{\lambda}} \left( H^{\star}_{1}(\A^{\textrm{T}} \bm{\lambda} )  
+ H^{\star}_{2}(\B^{\textrm{T}} \bm{\lambda} ) - \left< \bm{\lambda}, \cbf \right> \right), 
\label{eq:Ldualproblem}
\end{align}
\end{subequations}
where $\bm{\lambda} \in \mathbb{R}^{m}$ is a dual variable,
$^{\textrm{T}}$ denotes the transposition, and $H^{\star}_{i}: \mathbb{R}^{m} \rightarrow \mathbb{R} \cup \{ \infty \}$ 
is the convex conjugate (the Legendre transformation for the scalar case) of $H_{i}$ \cite{fenchel1949conjugate} as
\begin{align}
& H^{\star}_{1}( \A^{\textrm{T}} \bm{\lambda} ) = \sup_{\p} \left( \left< \bm{\lambda}, \A \p \right> - H_{1}(\p) \right), 
\label{eq:dualconjugate1} \\
& H^{\star}_{2}(\B^{\textrm{T}} \bm{\lambda} ) = \sup_{\q} \left( \left< \bm{\lambda}, \B \q \right> - H_{2}(\q) \right). 
\label{eq:dualconjugate2}
\end{align}
By replacing $\bm{\lambda} \rightarrow \w$, and setting $G_{1}(\w)
= H^{\star}_{1}(\A^{\textrm{T}} \w ), 
G_{2}(\w) = H^{\star}_{2}(\B^{\textrm{T}} \w ) - \left< \bm{\lambda}, \cbf \right> $, the Lagrangian dual ascent problem, which reformulates \eqref{eq:Ldualproblem_ascent} into the minimization (\ref{eq:Ldualproblem}), 
is of the form (\ref{eq:mintwoccp}).

\subsection{Bregman Monotone Operator Splitting (B-MOS)}
\label{sec:2_2}

In this section, we generalize MOS solvers with the aim to obtain faster convergence. 
MOS methods have been studied as solvers for the problem
(\ref{eq:subdif_zero}) and are summarized well in e.g., \cite{ryu2016primer, bauschke2017convex}. 
In the present paper, we focus on three well-known monotone operator splitting methods: 
namely Peaceman-Rachford (P-R) splitting \cite{peaceman1955numerical}, Douglas-Rachford
(D-R) splitting \cite{douglas1956numerical} and forward-backward (F-B) splitting \cite{passty1979ergodic}.

We generalize the conventional Euclidean distance metric used in MOS methods to the Bregman divergence (B-MOS). The motivation for the generalization is that our Bregman divergence based approach can be used to obtain significantly faster convergence rates than conventional Euclidean distance based solvers. (We will explain how to obtain fast convergence rates on the basis of Bregman divergence in Sec. \ref{sec:2_3} and \ref{sec:2_4}. )

We first define Bregman divergence together with a property relevant in the present context. The Bregman divergence of
a first point $\w \in \mathbb{R}^{m} $ and a second point $\z \in  \mathbb{R}^{m}$ \cite{bregman1967relaxation} is defined as
\begin{align}
B_{D}(\w \hspace{-2pt} \parallel \hspace{-2pt} \z) = D(\w)  -  D(\z) -  \left< \nabla D(\z), \w - \z \right>,
\label{eq:bregman}
\end{align}
where $\nabla$ denotes the gradient operator and where the definition of Bregman
divergence allows any continuously differentiable strictly convex function for $D$, e.g., \cite{boyd2004convex}. 
An important property for $D$ is that if it is limited to satisfy $\nabla D (\bm{0})
\hspace{-2pt}  = \hspace{-2pt} \bm{0}$, which is equivalent to $\nabla D^{-1} (\bm{0}) \hspace{-1pt} = \hspace{-1pt} \bm{0}$,  
then the fixed point specified by \eqref{eq:subdif_zero} is unaffected by the application of $\nabla D^{-1}$:
\begin{align}
&\bm{0} \in \partial G_{1} (\w^{\ast} ) + \partial G_{2}  (\w^{\ast} ), \notag \\
&\nabla D^{-1} ( \bm{0} ) \in \nabla D^{-1} \left( \partial G_{1} (\w^{\ast} ) + \partial G_{2}  (\w^{\ast} ) \right), \notag \\
&\bm{0} \in \nabla D^{-1} \partial G_{1} (\w^{\ast} ) + \nabla D^{-1} \partial G_{2}  (\w^{\ast} ), 
\hspace{10pt} ( \textrm{if} \hspace{3pt} \nabla D ( \bm{0} ) = \bm{0} ). 
\label{eq:D0_requirements} 
\end{align}
Note that for $D(\w) \hspace{-2pt} = \hspace{-2pt} \frac{1}{2 \kappa} \| \w \|_{2}^{2}
\hspace{2pt} (\kappa \hspace{-2pt} > \hspace{-2pt} 0)$, 
the Bregman divergence reduces to the Euclidean distance.

Next, we define a number of operators that we will need below. Some of
these operators are well-known. We define the $D$-forward step as 
\begin{align}
F_{i} &= I - \nabla D^{-1} \partial G_{i}, 
\label{ref:dforwardstep}
\end{align}
where $\nabla D^{-1}$ is applied to the subgradient operator $\partial G_{i}$ to modify the metric of variable space. 
Furthermore, we define the $D$-resolvent operator \cite{bauschke2003bregman} (or $D$-backward step) 
$R_{i} \hspace{3pt} (i \hspace{-2pt} = \hspace{-2pt} 1,2)$, 
the new $D$-Cayley operator $C_{i}\hspace{3pt} (i \hspace{-2pt} = \hspace{-2pt} 1,2)$, 
and the averaged operator $A_{J}$ with respect to the operator $J$, e.g., \cite{ryu2016primer, bauschke2017convex}, as
\begin{align}
R_{i} &= (I + \nabla D^{-1} \partial G_{i})^{-1}
= (\nabla D + \partial G_{i})^{-1} \nabla D, 
\label{ref:dResolventoperator_reform} \\
%
%
%
C_{i} &= R_{i} F_{i} \notag \\
&= (I + \nabla D^{-1} \partial G_{i})^{-1} (I - \nabla D^{-1} \partial G_{i}) \label{ref:dCayleyoperator1} \\
&= 2 (I + \nabla D^{-1} \partial G_{i})^{-1} - (I + \nabla D^{-1} \partial G_{i})^{-1} (I + \nabla D^{-1} \partial G_{i}) \notag \\
&= 2 (I + \nabla D^{-1} \partial G_{i})^{-1} - I \notag \\
&= 2 R_{i} - I, 
\label{ref:dCayleyoperator2} \\
A_{J} &= (1-\alpha) I + \alpha J = I + \alpha ( J - I ),
\label{ref:averagedoperator}
\end{align}
where $\alpha \hspace{-1pt} \in \hspace{-1pt} (0,1)$. 
When the Bregman divergence is based on the Euclidean distance, i.e., 
$D(\w) = \hspace{-1pt} \frac{1}{2 \kappa} \| \w \|_{2}^{2}$, 
the well-known (Euclidean) resolvent operator is then obtained: 
$R_{i} = (I + \kappa \partial G_{i} )^{-1}$, 
and the metric of $F_{i}$ and $C_{i}$ is then also Euclidean (e.g., \cite{ryu2016primer}). 
The properties of $R_{i}, C_{i}$ and $F_{i}$ are investigated in more detail in Appendix \ref{sec:appendix}.

With the above definitions and properties, we are now able to derive Peaceman-Rachford splitting generalized using Bregman divergence 
(Bregman Peaceman-Rachford splitting) by reformulating \eqref{eq:D0_requirements}, which is the fixed-point condition 
assuming that $\nabla D ( \bm{0} ) \hspace{-2pt} = \hspace{-2pt} \bm{0}$, as 
\begin{align}
\mathbf{0} &\in \nabla D^{-1} \partial G_{2} (\w) +  \nabla D^{-1} \partial G_{1} (\w), \notag \\
\mathbf{0} &\in ( I + \nabla D^{-1} \partial G_{2} ) (\w) - ( I - \nabla D^{-1} \partial G_{1} ) (\w), 
\label{eq:fixedpoint00}
\end{align}
where $I$ and $^{-1}$ are the identity operator and the inverse operator, respectively. Since $D$ is a strictly convex function, $\nabla D$ and its inverse $\nabla D^{-1}$ are monotone operators that have a unique relation between input and output vectors. 
By setting $\w \hspace{-2pt} \in
\hspace{-2pt} R_{1} (\z)$, the fixed point condition \eqref{eq:fixedpoint00} can be written as
\begin{align}
\mathbf{0} &\in ( I + \nabla D^{-1} \partial G_{2} ) R_{1} (\z) - ( I - \nabla D^{-1} \partial G_{1} ) R_{1} (\z), \notag \\
\mathbf{0} &\in R_{1} (\z) - R_{2} C_{1} (\z), \notag \\
\mathbf{0} &\in \frac{1}{2} (C_{1} + I) (\z) - \frac{1}{2} (C_{2} + I) C_{1} (\z). \notag 
\end{align}
Hence, we obtain the condition for a fixed point
\begin{align}
\z \in C_{2} C_{1} (\z).
\label{eq:PRsplit}
\end{align}

\begin{algorithm}[t]
   \caption{Bregman Peaceman-Rachford Splitting}
   \label{alg:BFPOS_PR}
\begin{algorithmic}
\STATE
{Initialization of $\z^{0}$} \\
\FOR{$t=0, \ldots, T - 1$}
\STATE
$\w^{t+1} = \arg \min_{\w} \left( G_{1}(\w) + B_{D}(\w \hspace{-2pt} \parallel \hspace{-2pt} \z^{t}) \right)$, \\
$\x^{t+1} = 2 \w^{t+1} - \z^{t}$, \\
$\y^{t+1} = \arg \min_{\y} \left( G_{2}(\y) + B_{D}(\y \hspace{-2pt} \parallel \hspace{-2pt} \x^{t+1})  \right)$, \\
$\z^{t+1} = 2 \y^{t+1} - \x^{t+1}$
\ENDFOR
\end{algorithmic}
\end{algorithm}
%

Appendix \ref{sec:appendix} shows that the $D$-Cayley operator $C_i$ is nonexpansive, i.e., 
it is Lipschitz continuous with the Lipschitz constant 1. 
The iterative application of \eqref{eq:PRsplit} generates a Cauchy sequence, and the
iterations follow Banach-Picard fixed-point iterations, e.g., \cite{berinde2007iterative}.

The iteration specified by \eqref{eq:PRsplit} can be decomposed into simpler steps by introducing additional auxiliary variables $\x \hspace{-1pt} \in \hspace{-1pt} \mathbb{R}^{m}$ and $\y \hspace{-1pt} \in \hspace{-1pt} \mathbb{R}^{m}$:
\begin{align}
&\hspace{-6pt} \w^{t+1} \hspace{-1pt} = \hspace{-1pt} R_{1} (\z^{t}) \hspace{-1pt} = \hspace{-1pt} 
(I + \nabla D^{-1} \partial G_{1} )^{-1} (\z^{t}), \hspace{-2pt}
\label{eq:PR01} \\
&\hspace{-6pt} \x^{t+1} \hspace{-1pt} = \hspace{-1pt} C_{1} (\z^{t}) \hspace{-1pt} = \hspace{-1pt} 
(2 R_{1} \hspace{-1pt} - \hspace{-1pt} I ) (\z^{t}) \hspace{-1pt} = \hspace{-1pt} 2 \w^{t+1} \hspace{-1pt} - \hspace{-1pt} \z^{t},
\hspace{-2pt}
\label{eq:PR02} \\
&\hspace{-6pt} \y^{t+1} \hspace{-1pt} = \hspace{-1pt} R_{2} (\x^{t+1}) \hspace{-1pt} = \hspace{-1pt} 
(I + \nabla D^{-1} \partial G_{2} )^{-1} (\x^{t+1}), \hspace{-2pt}
\label{eq:PR03} \\
&\hspace{-6pt} \z^{t+1} \hspace{-1pt} = \hspace{-1pt} C_{2} (\x^{t+1}) \hspace{-1pt} = \hspace{-1pt} 
(2 R_{2} \hspace{-1pt} - \hspace{-1pt} I ) (\x^{t+1}) \hspace{-1pt} = \hspace{-1pt} 2 \y^{t+1} \hspace{-1pt} - \hspace{-1pt} \x^{t+1}, \hspace{-2pt}
\label{eq:PR04}
\end{align}
where (\ref{eq:PR01}) corresponds to the Bregman proximal point algorithm, e.g., \cite{eckstein1993nonlinear}. 
This can be seen by first writing
\begin{align}
\w &\in R_{1} (\z) , \notag \\
\w &\in ( I + \nabla D^{-1} \partial G_{i} )^{-1} (\z) , \notag \\
( I + \nabla D^{-1} \partial G_{1} ) (\w) &\in \z, \notag \\
\mathbf{0} &\in \nabla D^{-1} \partial G_{1}(\w) + \w - \z, \notag \\
\mathbf{0} &\in \partial G_{1}(\w) + \nabla D(\w) - \nabla D(\z). 
\label{eq:bregmanproximal0}
\end{align}
Assuming the minimum exists, then the integral of \eqref{eq:bregmanproximal0} gives
\begin{align}
\w^{t+1} &=   \arg \min_{\w} \left( G_{1}(\w) + B_{D}(\w \hspace{-2pt} \parallel \hspace{-2pt} \z^{t}) \right). 
\label{eq:bregmanresolvent_integral1}
\end{align}
From \eqref{eq:bregmanresolvent_integral1}, 
we see that the metric of the cost function is generalized by using the Bregman divergence. 
By using \eqref{ref:dCayleyoperator2}, the variable update using the $D$-Cayley operator
can be obtained with \eqref{eq:PR02}. However, to show that the  
 update cost is based on the Bregman divergence, we rewrite it with another formulation.  By using
\eqref{ref:dCayleyoperator1}, the update procedure $\x \in C_{1}(\z)$ can be reformulated as  
\begin{align}
\x &\in (I + \nabla D^{-1} \partial G_{1} )^{-1} (I - \nabla D^{-1} \partial G_{1} ) (\z), \notag \\
(I + \nabla D^{-1} \partial G_{1} ) (\x) &\in (I - \nabla D^{-1} \partial G_{1} ) (\z), \notag \\
\mathbf{0} & \in \x - \z + \nabla D^{-1} \partial G_{1}(\x) + \nabla D^{-1} \partial G_{1}(\z), \notag \\
\mathbf{0} & \in \nabla D(\x) - \nabla D(\z) + \partial G_{1}(\x) + \partial G_{1}(\z). 
\label{eq:bregmanCayleyimp0}
\end{align}
Assuming that the minimum exists, then the integral of \eqref{eq:bregmanCayleyimp0} gives
\begin{align}
\x^{t+1} &=  \arg \min_{\x} \left( G_{1}(\x) + G_{1}(\z^{t}) + \left< \partial G_{1}(\z^{t}), \x - \z^{t} \right> + B_{D}(\x \hspace{-1pt} \parallel \hspace{-1pt} \z^{t}) \right). 
\label{eq:bregmanCayley1}
\end{align}
\eqref{eq:bregmanCayley1} also shows that the cost metric is generalized to a Bregman divergence. 
However, since the vector update with this procedure gives the same result as \eqref{eq:PR02}, 
we use the simple form \eqref{eq:PR02} for the implementation of the $D$-Cayley operator hereafter.  
The resulting Bregman Peaceman-Rachford splitting algorithm is summarized in \textbf{Algorithm} \ref{alg:BFPOS_PR}.

\begin{algorithm}[t]
   \caption{Bregman Douglas-Rachford Splitting}
   \label{alg:BFPOS_DR}
\begin{algorithmic}
\STATE
{Initialization of $\z^{0}$} \\
\FOR{$t=0, \ldots, T - 1$}
\STATE
$\w^{t+1} = \arg \min_{\w} \left( G_{1}(\w) + B_{D}(\w \hspace{-2pt} \parallel \hspace{-2pt} \z^{t}) \right)$, \\
$\x^{t+1} = 2 \w^{t+1} - \z^{t}$, \\
$\y^{t+1} = \arg \min_{\y} \left( G_{2}(\y) + B_{D}(\y \hspace{-2pt} \parallel \hspace{-2pt} \x^{t+1}) \right)$, \\
$\z^{t+1} = \z^{t} + 2 \alpha (\y^{t+1} - \w^{t+1})$
\ENDFOR
\end{algorithmic}
\end{algorithm}

Bregman Douglas-Rachford splitting, a generalization of Douglas-Rachford splitting,
is obtained by introducing the averaged operator into (\ref{eq:PRsplit}):
\begin{align}
\z  &\in  \alpha C_{2} C_{1} (\z) + (1 - \alpha) \z, 
\label{eq:DRsplit} \\
\z &\in A_{C_{2} C_{1}} (\z). 
\label{eq:DRsplit2}
\end{align}
\eqref{eq:DRsplit2} can be decomposed into
(\ref{eq:PR01})-(\ref{eq:PR03}), augmented by
\begin{align}
&\z^{t+1} = \alpha (2 \y^{t+1} - \x^{t+1}) + (1 - \alpha) \z^{t} = \z^{t} + 2 \alpha (\y^{t+1} - \w^{t+1}).
\label{eq:DR04}
\end{align}
When $J$ is a nonexpansive operator, 
$A_{J}$ is also a nonexpansive operator, e.g., \cite{bauschke2017convex}. 
Therefore, the Bregman Douglas-Rachford splitting algorithm is  Banach-Picard fixed point iteration. 
Its update rule is summarized in \textbf{Algorithm} \ref{alg:BFPOS_DR}.

%
\begin{algorithm}[t]
   \caption{Bregman Forward-Backward Splitting}
   \label{alg:BFPOS_FB}
\begin{algorithmic}
\STATE
{Initialization of $\w^{0}$} \\
\FOR{$t=0, \ldots, T - 1$}
\STATE
$\w^{t+1} = \arg \min_{\w} \left( G_{2}(\w) + G_{1}(\w^{t})  + \left< \partial G_{1}(\w^{t}), \w - \w^{t} \right> + B_{D}(\w \hspace{-1pt} \parallel \hspace{-1pt} \w^{t}) \right)$
\ENDFOR
\end{algorithmic}
\end{algorithm}

%
%
%

Apart from the Banach-Picard fixed point iterations, 
reformulating (\ref{eq:subdif_zero}) leads forward-backward splitting generalized using the Bregman divergence (Bregman forward-backward splitting):
\begin{align}
\mathbf{0} &\in \nabla D^{-1} \partial G_{2} (\w) + \nabla D^{-1} \partial G_{1} (\w), 
\notag \\
\mathbf{0} &\in (I + \nabla D^{-1} \partial G_{2}) (\w) - (I - \nabla D^{-1} \partial G_{1}) (\w), 
\notag \\
( I  + \nabla D^{-1} \partial G_{2}) (\w) &\in (I  - \nabla D^{-1} \partial G_{1}) (\w), 
\notag \\
\w &\in (I + \nabla D^{-1} \partial G_{2})^{-1}
(I  -  \nabla D^{-1} \partial G_{1}) (\w), 
\notag \\
\w &\in R_{2} F_{1} (\w). 
\label{eq:FBsplit}
\end{align}
Therefore, the update procedure is given by 
\begin{align}
&\x^{t+1} = F_{1} (\w^{t}) = ( I - \nabla D^{-1} \partial G_{1} ) (\w^{t}), \label{eq:FBsplitting_01} \\ 
&\w^{t+1} = R_{2} (\x^{t+1}). 
\label{eq:FBsplitting_02}
\end{align}
The procedure \eqref{eq:FBsplitting_01}-\eqref{eq:FBsplitting_02} can be summarized by
\begin{align}
\w^{t+1} &= R_{2} F_{1} (\w^{t}) \notag \\
&= \arg \min_{\w} \hspace{-2pt} \left( G_{2}(\w) \hspace{-2pt} + \hspace{-2pt} G_{1}(\w^{t}) \hspace{-2pt} + \hspace{-2pt} \left< \partial G_{1}(\w^{t}), \w \hspace{-2pt} - \hspace{-2pt} \w^{t} \right> \hspace{-2pt} + \hspace{-2pt} B_{D}(\w \hspace{-3pt} \parallel \hspace{-3pt} \w^{t}) \right)\hspace{-2pt}. 
\label{eq:FBsplitting3}
\end{align}
The Bregman forward-backward splitting algorithm is summarized in \textbf{Algorithm} \ref{alg:BFPOS_FB}. 
For the MOS algorithms we derived, the metric generalization using Bregman divergence was achieved by replacing 
$\partial G_{i}$ by $\nabla D^{-1} \partial G_{i}$. 
The metric of other MOS algorithms such as forward-backward-forward splitting \cite{tseng2000modified} 
and Davis-Yin three-operator splitting \cite{davis2017three} can be similarly generalized using Bregman divergence. 
Since this does not affect our main conclusions, their derivations are not described in this paper. 
%

%
%
%

In this subsection, several MOS algorithms were generalized using the Bregman divergence (B-MOS). 
To exploit this generalization and make the algorithms converge faster, an appropriate Bregman divergence must be designed. 
The step for designing an appropriate metric of it is provided in the next subsection.

\subsection{Bregman Divergence Design for Fast Convergence Rate}
\label{sec:2_3}

We now introduce the main idea of how to design Bregman divergence for fast convergence. As explained in Sec. \ref{sec:2_2}, the metric of variable space was generalized by using Bregman divergence instead of the Euclidean distance used in the traditional MOS solvers. We first investigate how the cost property is modified by applying $\nabla D^{-1}$ to $\partial G_{i}$ because this will provide us with an indication on how to design Bregman divergence for fast convergence.

As illustrated in Appendix \ref{sec:appendix}, we assume that the properties of $G_{i}$ are represented by using any different two points $\w$ and $\z$, given
\begin{align}
\gamma_{\textrm{LB}, i}
\| \w \hspace{-1pt} - \hspace{-1pt} \z  \|_{2}
\leq 
\| \partial G_{i}(\w) - \partial G_{i}(\z) \|_{2}
\leq 
\gamma_{\textrm{UB}, i} \| \w \hspace{-1pt} - \hspace{-1pt} \z \|_{2},
\label{eq:G_2bounds_sec2}
\end{align}
where $0 \hspace{-2pt} \leq \hspace{-2pt} \gamma_{\textrm{LB}, i} \hspace{-2pt} \leq \hspace{-2pt} \gamma_{\textrm{UB}, i} \hspace{-2pt} < \hspace{-2pt} +\infty$. 
Applying $\nabla D^{-1}$ to $\partial G_{i}$, as in \eqref{ref:dforwardstep}, \eqref{ref:dResolventoperator_reform}, \eqref{ref:dCayleyoperator2}, modifies the properties of $G_{i}$ to
\begin{align}
\sigma_{\mathrm{LB}, i} \| \w \hspace{-2pt} - \hspace{-2pt} \z \|_{2}
\leq 
\| \nabla D^{-1}  \partial G_{i}(\w) - \nabla D^{-1} \partial G_{i}(\z)  \|_{2} 
\leq 
\sigma_{\mathrm{UB}, i} \| \w \hspace{-1pt} - \hspace{-1pt} \z \|_{2},
\label{eq:G_4bounds_sec2}
\end{align}
where $0 \hspace{-2pt} \leq \hspace{-2pt} \sigma_{\mathrm{LB}, i} \hspace{-2pt} \leq \hspace{-2pt} \sigma_{\mathrm{UB}, i} \hspace{-2pt} < \hspace{-2pt} +\infty$. 
This indicates that $\nabla D^{-1} \partial G_{i}$ is assumed to be Lipschitz continuous, but is not to be strongly convex. 
By modifying $\nabla D$ while satisfying $\nabla D(\bm{0}) \hspace{-1pt} = \hspace{-1pt} \bm{0}$, the pair of $\{ \sigma_{\mathrm{UB}, i}, \sigma_{\mathrm{LB}, i} \}$ will be changed.

\begin{table}[t]
\centering
\caption{Convergence Rates of B-MOS Algorithms.}
\vspace{-5pt}
\label{tbl:crates}
\begin{tabular}{|l|l|}
\hline
Bregman Peaceman-Rachford splitting & $\parallel \hspace{-1pt} \z^{t} \hspace{-1pt} - \hspace{-1pt} \z^{\ast} \hspace{-1pt} \parallel_{2}
\leq \hspace{-1pt} 
(\eta_{1} \eta_{2})^{t}
\parallel \hspace{-1pt} \z^{0} \hspace{-1pt} - \hspace{-1pt} \z^{\ast} \hspace{-1pt} \parallel_{2} $
 \\ \hline
Bregman Douglas-Rachford splitting & 
$\parallel \hspace{-2pt} \z^{t} \hspace{-1pt} - \hspace{-1pt} \z^{\ast} \hspace{-2pt} \parallel_{2}
\leq \hspace{-1pt} 
(1 \hspace{-1pt} - \hspace{-1pt} \alpha \hspace{-1pt} + \hspace{-1pt} \alpha \eta_{1} \eta_{2})^{t}
 \hspace{-2pt}
\parallel \hspace{-2pt} \z^{0} \hspace{-1pt} - \hspace{-1pt} \z^{\ast} \hspace{-2pt} \parallel_{2}$ 
  \\ \hline
Bregman Forward-Backward splitting &  
$\parallel \hspace{-2pt} \w^{t} - \w^{\ast} \hspace{-2pt} \parallel_{2} 
\leq \lambda^{t}
\parallel \hspace{-2pt} \w^{0} - \w^{\ast} \hspace{-2pt} \parallel_{2} $
\\ \hline
\end{tabular}
\vspace{-0pt}
\end{table}

To clarify the optimal convergence condition associated with $\{ \sigma_{\mathrm{UB}, i}, \sigma_{\mathrm{LB}, i} \}$, 
the convergence rates on B-MOS algorithms were investigated in Appendix \ref{sec:appendixB} 
and they are summarized in Table \ref{tbl:crates}. For the Bregman Peaceman-Rachford splitting \eqref{eq:PRsplit}, 
the convergence rate is predicted by
\begin{align}
\parallel \hspace{-1pt} \z^{t} \hspace{-1pt} - \hspace{-1pt} \z^{\ast} \hspace{-1pt} \parallel_{2}
\leq \hspace{-1pt} 
(\eta_{1} \eta_{2})^{t}
\parallel \hspace{-1pt} \z^{0} \hspace{-1pt} - \hspace{-1pt} \z^{\ast} \hspace{-1pt} \parallel_{2}, 
\label{eq:BPRsplit_crate_sec2}
\end{align}
where $\z^{\ast}$ denotes the fixed point of $\z$ and
\begin{align}
\eta_{i} = \sqrt{ 1 - \frac{4 \sigma_{\mathrm{LB}, i}}{ (1 + \sigma_{\mathrm{UB}, i} )^{2} }}.
\label{eq:eta_sec2}
\end{align}
\eqref{eq:BPRsplit_crate_sec2} indicates that fast convergence will be achieved by
modifying $\{ \sigma_{\mathrm{UB}, i}, \sigma_{\mathrm{LB}, i} \}$ such that $\eta_{i}$ is zero. 
For the Bregman Forward-Backward splitting \eqref{eq:FBsplit}, 
the convergence rate is described by
\begin{align}
\parallel \hspace{-2pt} \w^{t} - \w^{\ast} \hspace{-2pt} \parallel_{2} 
\leq \lambda^{t}
\parallel \hspace{-2pt} \w^{0} - \w^{\ast} \hspace{-2pt} \parallel_{2}, 
\label{eq:BFBsplit_crate_sec2}
\end{align}
where
\begin{align}
\lambda = 
\sqrt { \frac{ 1 - 2 \sigma_{\textrm{LB}, 1} + \sigma_{\textrm{UB}, 1}^{2} }{ (1 + \sigma_{\textrm{LB}, 2} )^{2} } },
\label{eq:lambda_sec2}
\end{align}
must be reduced to zero for fast convergence rate. 
As noted in Appendix \ref{sec:appendix}, the convergence rate factors for B-MOS algorithms are optimized as $\eta_{i} = 0, \lambda = 0$ only if $\{ \sigma_{\mathrm{UB}, i}, \sigma_{\mathrm{LB}, i} \}$ satisfies: 
\begin{align}
\sigma_{\mathrm{LB}, i} = 1, \sigma_{\mathrm{UB}, i} = 1. 
\label{optimalsigmapair_sec2} 
\end{align}
Thus, from \eqref{optimalsigmapair_sec2}, 
we conclude that, for fast convergence the Bregman divergence must be designed such that
both $\sigma_{\mathrm{UB}, i}$ and $\sigma_{\mathrm{LB}, i}$ approach 1.

Substituting \eqref{optimalsigmapair_sec2} into \eqref{eq:G_4bounds_sec2} illustrates the meaning of \eqref{optimalsigmapair_sec2}, we have 
\begin{align}
\| \nabla D^{-1}  \partial G_{i}(\w) - \nabla D^{-1} \partial G_{i}(\z) \|_{2} 
\approx 
\| \w - \z \|_{2}. 
\label{eq:optimalConvexity}
\end{align}
\eqref{eq:optimalConvexity} indicates that $\nabla D$ modifies the convexity of $\partial G_{i}$ to be proportional in the $L_{2}$ norm domain. 
It is equivalent to modifying the metric of the space to make $G_{i}$ a quadratic function with a Hessian that is a unit matrix $\frac{1}{2} \hspace{-2pt} \parallel \hspace{-2pt} \w \hspace{-2pt} - \hspace{-2pt} \w^{\ast} \hspace{-2pt} \parallel_{2}^{2}$.

\subsection{Implementation Example of Bregman Divergence}
\label{sec:2_4}

We now discuss a practical Bregman divergence design method that approximates \eqref{eq:optimalConvexity}. 
In order to make this method available even if $G_{i}$ is not differentiable strictly convex, 
it is assumed that we have a differentiable strictly convex function $\overline{G}_{i}$ that approximates  $G_{i}$. (A method to obtain $\overline{G}_{i}$ is provided later in this section. )
When we have $\overline{G}_{i}$, it is a good choice to follow \eqref{eq:optimalConvexity} at the first setting $t \hspace{-1pt} = \hspace{-1pt} 0$ as
\begin{align}
\nabla D (\w) \hspace{-1pt} = \hspace{-1pt} \nabla \overline{G}_{1} (\w) \hspace{-1pt} - \hspace{-1pt} \nabla \overline{G}_{1} (\bm{0}),
\label{eq:nablaD_Design_1}
\end{align}
where the subtractive term is used to satisfy $\nabla D(\bm{0}) \hspace{-1pt} = \hspace{-1pt} \bm{0}$. 
However, our overall cost is $G \hspace{-1pt} = \hspace{-1pt} G_{1} \hspace{-1pt} + \hspace{-1pt} G_{2}$ and $D$ is restricted to be differentiable strictly convex. 
As a simple design of $D$ that works even when the convexity property of $G_{1}$ and that of $G_{2}$ are quite different, 
we use a $\nabla D$ that matches $\nabla \overline{G}$ as
\begin{align}
\nabla D (\w) \hspace{-1pt} = \hspace{-1pt} 
\nabla \overline{G} (\w)
- \nabla \overline{G} (\bm{0}). 
\label{eq:nablaD_Design_2}
\end{align}
The integral of \eqref{eq:nablaD_Design_2} is given by
\begin{align}
D (\w) \hspace{-1pt} = \hspace{-1pt} 
\overline{G} (\w) 
- \left< \nabla \overline{G} (\bm{0}), \w \right> - \overline{G} (\bm{0}). 
\label{eq:D_Design_2}
\end{align}
Although better choices for $D$, which match \eqref{eq:optimalConvexity} better, are likely possible, we leave that for future work. This is because it would be dependent on the combination of the convexity property of $G_{1}$ and that of $G_{2}$.

As a design of $\overline{G}$, we use a quadratic representation of $G$. When $G$ is differentiable at the point $\z$, 
a second-order Taylor expansion around that point is a choice of $\overline{G}$ as
\begin{align}
\overline{G}(\w) = G(\z) + \left< \nabla G(\z), \w - \z \right> + \frac{1}{2} \left< \M(\z) (\w - \z), \w - \z \right>, 
\label{eq:Gover_2GD}
\end{align}
where $G$ is allowed to be replaced by its majorization function $G(\w) + \epsilon/2  \| \w \hspace{-1pt} - \hspace{-1pt} \z \|_{2}^{2} \hspace{3pt} (\epsilon \hspace{-2pt} > \hspace{-2pt} 0)$ when it is not differentiable at the point $\z$ 
and $\M(\z)$ denotes the Hessian of $G$ or its majorization function. 
By substituting \eqref{eq:Gover_2GD} into \eqref{eq:D_Design_2}, we obtain 
\begin{align}
D (\w) &= \frac{1}{2} \left< \M(\z) \w, \w \right>. 
\label{eq:D_2GD}
\end{align}
This indicates that the Bregman divergence is given by
\begin{align}
B_{D}^{(\textrm{Newton})}(\w \hspace{-2pt} \parallel \hspace{-2pt} \z) = \frac{1}{2} \left< \M(\z) (\w - \z), (\w - \z) \right>. 
\label{eq:bregman_Taylor_2GD}
\end{align}
Since the metric of variable space is modified by using a Hessian matrix, the Bregman divergence design \eqref{eq:bregman_Taylor_2GD} is associated with the Newton method. 
Note that this Bregman divergence design is not perfectly matched with the property \eqref{eq:optimalConvexity} 
because a second-order approximation is used in \eqref{eq:D_2GD}. 
Following \eqref{eq:G_4bounds_sec2}, the properties of $G$ are then modified by using 
$\nabla D^{-1} \hspace{-2pt} = \hspace{-2pt} \M^{-1}(\z)$ as 
\begin{align}
\sigma_{\mathrm{LB}, i}^{\textrm{(Newton)}} \| \w \hspace{-2pt} - \hspace{-2pt} \z \|_{2}
\leq 
\| \nabla D^{-1}  \partial G_{i}(\w) - \nabla D^{-1} \partial G_{i}(\z)  \|_{2} 
\leq 
\sigma_{\mathrm{UB}, i}^{\textrm{(Newton)}} \| \w \hspace{-1pt} - \hspace{-1pt} \z \|_{2},
\label{eq:G_sigmarange_2GD}
\end{align}
where both $\sigma_{\mathrm{UB}, i}^{\textrm{(Newton)}}$ and $\sigma_{\mathrm{LB}, i}^{\textrm{(Newton)}}$ would approach 1.

When $\M(\z)$ in \eqref{eq:bregman_Taylor_2GD} is replaced by its diagonalized matrix $\L(\z)$, 
the Bregman divergence is then given by
\begin{align}
B_{D}^{(\textrm{AGD})}(\w \hspace{-2pt} \parallel \hspace{-2pt} \z) = \frac{1}{2} \left< \L(\z) (\w - \z), (\w - \z) \right>, 
\label{eq:bregman_Taylor_AGD}
\end{align}
where the diagonal elements of $\L(\z)$ are the same as $\M(\z)$. 
The Bregman divergence form \eqref{eq:bregman_Taylor_AGD} is associated with the accelerated gradient descent (AGD) 
because its step-size is independent for each element. 
For smoothly variable update, it is often used to update $\L(\z)$ using it at the previous step \cite{duchi2011adaptive, kingma2014adam, tieleman2012lecture}. 
%
%

Finally, the relationship between the conventional (Euclidean) MOS solvers and the first-order gradient descent is briefly discussed. 
As noted in Sec. \ref{sec:2_2}, the Bregman divergence reduces to the Euclidean distance when using $D(\w) \hspace{-2pt} = \hspace{-2pt} \frac{1}{2 \kappa} \| \w \|_{2}^{2}$ as
\begin{align}
B_{D}^{(\textrm{GD})}(\w \hspace{-2pt} \parallel \hspace{-2pt} \z) = \frac{1}{2 \kappa} \| \w - \z \|_{2}^{2}. 
\label{eq:bregman_Taylor_GD}
\end{align}
Since the variable is then updated with its gradient multiplied to a given step-size, 
this is associated with the gradient descent (GD) method.

It is difficult to provide a 
model of the eigenvalue dynamic range differences for the three methods.
It is reasonable to assume that, in general,
\begin{align}
1 \leq 
\frac{\sigma_{\mathrm{UB}, i}^{\textrm{(Newton)}}}{\sigma_{\mathrm{LB}, i}^{\textrm{(Newton)}}} 
\leq 
\frac{\sigma_{\mathrm{UB}, i}^{\textrm{(AGD)}}}{\sigma_{\mathrm{LB}, i}^{\textrm{(AGD)}}} 
\leq 
\frac{\sigma_{\mathrm{UB}, i}^{\textrm{(GD)}}}{\sigma_{\mathrm{LB}, i}^{\textrm{(GD)}}}. 
\label{eq:eigenratiomodel}
\end{align}
Moreover, it is reasonable to assume that $\sigma_{\mathrm{UB}, i}$ and $\sigma_{\mathrm{LB}, i}$ will be closest to 1 with the Newton method.

\section{APPLICATION EXAMPLE}
\label{sec:3}

In this section, several B-MOS solvers are applied to 
the total variation (TV) denoising problem \cite{rudin1992nonlinear} as an example. 
We first formulate the problem in Sec. \ref{sec:3_1} 
and its solver implementation is provided in Sec. \ref{sec:3_2}. 
Through numerical experiments in Sec. \ref{sec:3_3}, we will illustrate the effectiveness of B-MOS.

\subsection{Problem Definition}
\label{sec:3_1}

Let suppose that the observed source $\s \hspace{-2pt} \in \hspace{-2pt} \mathbb{R}^{m}$ including random noise $\e \hspace{-2pt} \in \hspace{-2pt} \mathbb{R}^{m}$ is given. 
When the original source is denoted by $\u^{\ast} \hspace{-2pt} \in \hspace{-2pt} \mathbb{R}^{m}$, 
the generative process of $\s$ is modeled by $\s \hspace{-2pt} = \hspace{-2pt} \u^{\ast} \hspace{-2pt} + \hspace{-2pt} \e$. 
TV denoising is used to remove noise from $\s$ and its cost function is formulated by
\begin{align}
\inf_{\u} \hspace{4pt} \frac{1}{2} \parallel \hspace{-2pt} \s - \u \hspace{-2pt} \parallel_{2}^{2} + \parallel \hspace{-2pt} \u \hspace{-2pt} \parallel_{\textrm{TV}}, 
\label{eq:problemdefine_mat_sec3}
\end{align}
where the TV norm \cite{rudin1992nonlinear} in the elastic net norm form \cite{zou2005regularization} is denoted by
\begin{align}
\parallel \hspace{-2pt} \u \hspace{-2pt} \parallel_{\textrm{TV}} \hspace{2pt} =  
\mu \left( \frac{\theta}{2} \| \bm{\Phi} \u \|_{2}^{2} + \hspace{-2pt} \| \bm{\Phi} \u \|_{1} \right), 
\label{eq:TVnorm_mat_sec3}
\end{align}
where $\mu \hspace{-2pt} > \hspace{-2pt} 0, \theta  \hspace{-2pt} > \hspace{-2pt} 0$ and 
$\bm{\Phi} \hspace{-2pt} \in \hspace{-2pt} \mathbb{R}^{m \times m}$ is full-rank and is used to calculate the discrete difference between neighborhood elements. For the case that a Sobel filter is used, the $i$-th element of $\bm{\Phi} \u$ is
\begin{align}
&\left[ \bm{\Phi} \u \right]_{i} = 
u_{i-1} - u_{i+1},
\end{align}
However since the lower case affine transformation is included in the regularization term, 
it may be difficult to update $\u$ such that it reduces the overall cost \eqref{eq:problemdefine_mat_sec3}.

To overcome this issue, applying MOS solvers is effective. 
The problem form \eqref{eq:problemdefine_mat_sec3} is reformulated 
by using an auxiliary variable $\v \hspace{-2pt} \in \hspace{-2pt} \mathbb{R}^{m}$ as 
\begin{align}
\inf_{ \u, \v } \hspace{4pt} 
H_{1}(\u) + H_{2}(\v) \hspace{15pt} \textrm{s.t.} \hspace{2pt} \v = \bm{\Phi} \u, 
\label{eq:problemdefine_vec_sec3}
\end{align}
where $H_{1}(\u) \hspace{-1pt} = \hspace{-1pt} \frac{1}{2}  \| \s \hspace{-1pt} - \hspace{-1pt} \u  \|_{2}^{2}$ and $H_{2}(\v) \hspace{-1pt} = \hspace{-1pt} \mu \left( \frac{\theta}{2} \| \v \|_{2}^{2} + \hspace{-1pt} \| \v \|_{1} \right)$. For the  linearly constrained problem \eqref{eq:problemdefine_vec_sec3}, 
it is usual to solve the Lagrangian dual ascent problem as explained in Ex. 2 of Sec. \ref{sec:2_1}. 
The associated Lagrangian is given by
\begin{align}
\mathcal{L}(\u, \v, \w) = H_{1}(\u) + H_{2}(\v) + \left< \w,  - \bm{\Phi} \u + \v \right>, 
\label{eq:lagrangian_sec3}
\end{align}
where $\w \hspace{-2pt} \in \hspace{-2pt} \mathbb{R}^{m}$ denotes the dual variable. 
Its dual problem is
\begin{align}
\sup_{ \w } 
\inf_{ \u, \v } \hspace{2pt}
\mathcal{L}(\u, \v, \w)
%
\hspace{-2pt} = \hspace{-2pt} 
- \inf_{ \w } \left( H^{\star}_{1} \left( \bm{\Phi}^{\textrm{T}} \w  \right) 
\hspace{-2pt} + \hspace{-2pt} H^{\star}_{2} ( - \w )  \right), \hspace{-5pt}
\label{eq:Ldualproblem_b_sec3}
\end{align}
%
where the convex conjugate of $H_{i} \hspace{2pt} (i \hspace{-2pt} = \hspace{-2pt} 1,2)$ is denoted by
\begin{align}
H^{\star}_{1} \left( \bm{\Phi}^{\textrm{T}} \w \right) 
&= \sup_{\u} \left( \left< \bm{\Phi}^{\textrm{T}} \w, \u \right> - H_{1}(\u) \right), 
\label{eq:dualconjugate1_sec3} \\
H^{\star}_{2}( -\w ) 
&= \sup_{ \v } \left( - \left< \w, \v \right> - H_{2}(\v) \right). 
\label{eq:dualconjugate2_sec3}
\end{align}
Since \eqref{eq:Ldualproblem_b_sec3} indicates that we will optimize $\w$ such that optimizes the sum of two CCP functions, 
the dual problem of TV denoising is of the form \eqref{eq:mintwoccp}. Hence, any B-MOS solver can be used.

\subsection{Solver Implementation}
\label{sec:3_2}

To solve the problem \eqref{eq:Ldualproblem_b_sec3}, 
nonexpansive Bregman Peaceman-Rachford (B-P-R) splitting and Bregman Douglas-Rachford (B-D-R) splitting are applied. 
To simplify notification, the subdifferential of the convex conjugate functions are denoted by 
$T_{1}(\w) \hspace{-2pt} = \hspace{-2pt} \bm{\Phi} \partial H^{\star}_{1} \left( \bm{\Phi}^{\textrm{T}} \w  \right)$ and 
$T_{2}(\w) \hspace{-2pt} = \hspace{-2pt} - \partial H^{\star}_{2} \left( - \w  \right)$, respectively. 
By using the results of Sec. \ref{sec:2_2}, the update procedure becomes 
\begin{align}
&\w^{t+1} \hspace{-1pt} = \hspace{-1pt} R_{1} (\z^{t})
\hspace{-1pt} = \hspace{-1pt} \left( I + \nabla D^{-1}  T_{1} \right)^{-1} (\z^{t}),
\label{eq:AP1_1_01} \\
&\x^{t+1} \hspace{-1pt} = \hspace{-1pt} C_{1} (\z^{t}) \hspace{-1pt} = \hspace{-1pt} 2 \w^{t+1} \hspace{-1pt} - \hspace{-1pt} \z^{t},
\label{eq:AP1_1_02} \\
&\y^{t+1} \hspace{-1pt} = \hspace{-1pt} R_{2} (\x^{t+1})
\hspace{-1pt} = \hspace{-1pt} ( I + \nabla D^{-1} T_{2} )^{-1} (\x^{t+1}), 
\label{eq:AP1_1_03} \\
&\z^{t+1} \hspace{-2pt} = \hspace{-2pt} 
\begin{cases}
C_{2} (\x^{t+1}) \hspace{-1pt} = \hspace{-1pt} 2 \y^{t+1} \hspace{-1pt} - \hspace{-1pt} \x^{t+1} 
& \textrm{(B-P-R splitting)} \\
\alpha C_{2} (\x^{t+1})  + (1-\alpha ) \z^{t}  \hspace{-1pt} = \hspace{-1pt} 
\z^{t} + 2 \alpha ( \y^{t+1} - \w^{t+1} )
& \textrm{(B-D-R splitting)} 
\end{cases}. 
\label{eq:AP1_1_04}
\end{align}
In the following, we will focus on the two remaining issues: 
(i) how to update the variables using the $D$-resolvent operator when its monotone operator is the subdifferential of the convex conjugate function as in \eqref{eq:AP1_1_01}, \eqref{eq:AP1_1_03} 
and (ii) the Bregman divergence design such that follows the discussion in Sec. \ref{sec:2_4}.

We now discuss the variable update using the $D$-resolvent operator $R_{1}$. 
Since the convex conjugate function includes the (primal) variable optimization of \eqref{eq:dualconjugate1_sec3}, its procedure forms an iterative update of $\{ \u, \w \}$. 
Associated with $H_{1}^{\star}$, let us consider the following problem:
\begin{align}
\inf_{\u} \hspace{2pt} H_{1}(\u) \hspace{15pt} 
\textrm{s.t. } \hspace{1pt} \bm{\Phi} \u = \bm{0}. 
\end{align}
We minimize the associated Lagrangian and update $\u$ accordingly. This minimization is equivalent to \eqref{eq:dualconjugate1_sec3}. 
For the associated Lagrangian $\mathcal{L}(\u, \w) \hspace{-2pt} =  \hspace{-2pt} H_{1}(\u)  \hspace{-2pt} -  \hspace{-2pt} \left< \w, \bm{\Phi} \u \right>$, $\u$ is updated such that minimizes it. Thus, the subgradient of it includes zero as 
\begin{align}
\bm{0} &\in \partial H_{1}(\u) - \bm{\Phi}^{\textrm{T}} \w,
\notag \\
\bm{0} &\in \u - \partial H_{1}^{-1} (\bm{\Phi}^{\textrm{T}} \w),
\notag \\
\bm{0} &\in \bm{\Phi} \u - \bm{\Phi} \partial H_{1}^{-1} (\bm{\Phi}^{\textrm{T}} \w).
\label{eq:u_represent1_0}
\end{align}
Since the inverse subdifferential of a CCP function is related to the subdifferential of its convex conjugate function \cite{tyrrell1970convex} as $T_{1}(\w) \hspace{-2pt} = \hspace{-2pt} \bm{\Phi} \partial H_{1}^{\star}(\bm{\Phi}^{\textrm{T}} \w) \hspace{-2pt} = \hspace{-2pt} \bm{\Phi} \partial H_{1}^{-1}(\bm{\Phi}^{\textrm{T}} \w)$, \eqref{eq:u_represent1_0} can be rewritten as reformulated as
\begin{align}
\bm{0} &\in \bm{\Phi} \u - T_{1} (\w), \notag \\
\bm{\Phi} \u &\in T_{1} \left( \w \right). 
\label{eq:u_represent1}
\end{align}
%
%
%
For the input/output pair of $D$-resolvent operator $\w \hspace{-2pt} \in \hspace{-2pt} R_{1}(\z)$, 
it is reformulated such that it includes $\{ \u, \w \}$ by inserting \eqref{eq:u_represent1} into \eqref{eq:AP1_1_01}:
\begin{align}
\w &\in ( I + \nabla D^{-1} T_{1} )^{-1} (\z), \notag \\
( I + \nabla D^{-1} T_{1} ) (\w) &\in \z, \notag \\
\w + \nabla D^{-1} ( \bm{\Phi} \u ) &= \z, \hspace{8pt}
\bm{0} \in T_{1}^{-1} ( \bm{\Phi} \u) - \w,  
\label{eq:w_u_represent}
\end{align}
where \eqref{eq:u_represent1} is used in \eqref{eq:w_u_represent}. 
By reorganizing \eqref{eq:w_u_represent}, it is found that $\{ \u, \z \}$ are related by
\begin{align}
&\bm{0} \in T_{1}^{-1}(\bm{\Phi} \u) - (\z - \nabla D^{-1} ( \bm{\Phi} \u ) ), \notag \\
&\bm{0} \in \bm{\Phi}^{\textrm{T}} T_{1}^{-1}(\bm{\Phi} \u) - \bm{\Phi}^{\textrm{T}} (\z - \nabla D^{-1} ( \bm{\Phi} \u ) ), \notag \\
&\bm{0} \in \partial H_{1}(\u) - \bm{\Phi}^{\textrm{T}} (\z - \nabla D^{-1} ( \bm{\Phi} \u ) ).
\label{eq:z_u_represent}
\end{align}
The integral of \eqref{eq:z_u_represent} gives a $\u$-update procedure using the dual auxiliary variable $\z^{t}$ as
\begin{align}
\u^{t+1} = \arg \min_{\u} \left( H_{1}(\u) - \left< \z^{t}, \bm{\Phi} \u \right> 
+ D^{-1} ( \bm{\Phi} \u ) \right). 
\label{eq:Dresolvent1_imp_001}
\end{align}
From \eqref{eq:w_u_represent}, the $\w$-update procedure using $\u^{t+1}$ is given by
\begin{align}
\w^{t+1} = \z^{t} - \nabla D^{-1} ( \bm{\Phi} \u^{t+1} ). 
\label{eq:Dresolvent1_imp_002}
\end{align}
In addition, for the update procedure using $R_{2}$ in \eqref{eq:AP1_1_03}, 
$\{ \v, \y \}$ are updated by 
\begin{align}
&\v^{t+1} \hspace{-1pt} = \hspace{-1pt} 
\arg \min_{\v} 
\left( H_{2}(\v) - \left< \x^{t+1}, -\v \right> + D^{-1}( - \v) \right), 
\label{eq:Dresolvent2_imp_001} \\
&\y^{t+1} \hspace{-1pt} = \hspace{-1pt}
\x^{t+1} - \nabla D^{-1} (- \v^{t+1}). 
\label{eq:Dresolvent2_imp_002}
\end{align}

By substituting the results in \eqref{eq:Dresolvent1_imp_001}--\eqref{eq:Dresolvent2_imp_002} into 
\eqref{eq:AP1_1_01}--\eqref{eq:AP1_1_04}, several dual auxiliary variables are removed, 
and the update procedure based on B-P-R and B-D-R splitting can be written as 
\begin{align}
&\u^{t+1} \hspace{-1pt} = \hspace{-1pt} 
\arg \min_{\u} 
\left( H_{1}(\u) - \left< \z^{t}, \bm{\Phi} \u \right> + D^{-1}(\bm{\Phi} \u) \right), 
\label{eq:AP1_2_01} \\
&\x^{t+1} \hspace{-1pt} = \hspace{-1pt}
\z^{t} - 2 \nabla D^{-1} (\bm{\Phi} \u^{t+1}),
\label{eq:AP1_2_02} \\
&\v^{t+1} \hspace{-1pt} = \hspace{-1pt} 
\arg \min_{\v} 
\left( H_{2}(\v) - \left< \x^{t+1}, -\v \right> + D^{-1}( - \v) \right), 
\label{eq:AP1_2_03} 
\\
&\z^{t+1} \hspace{-1pt} = \hspace{-1pt}
\begin{cases}
\x^{t+1} - 2 \nabla D^{-1} (-\v^{t+1}) 
& \textrm{(B-P-R splitting)}  \\
\z^{t} - 2 \alpha \left( \nabla D^{-1} (\bm{\Phi} \u^{t+1}) + \nabla D^{-1} (-\v^{t+1})  \right)
& \textrm{(B-D-R splitting)} 
\end{cases}.
\label{eq:AP1_2_04}
\end{align}
By substituting several nonlinearly transformed auxiliary variables
$\x^{t} \hspace{-2pt} = \hspace{-2pt} \nabla D^{-1}(\tilde{\x}^{t}) $ and $\z^{t} \hspace{-2pt} = \hspace{-2pt} \nabla D^{-1}(\tilde{\z}^{t})$ 
into \eqref{eq:AP1_2_01}--\eqref{eq:AP1_2_04}, 
a further simplified update procedure/notification is obtained. 
Then, \eqref{eq:AP1_2_01} is denoted by
\begin{align}
&\u^{t+1} \hspace{-1pt} = \hspace{-1pt} 
\arg \min_{\u} 
\left( H_{1}(\u) - \left< \nabla D^{-1}(\tilde{\z}^{t}), \bm{\Phi} \u \right> + D^{-1}(\bm{\Phi} \u) \right). 
\label{eq:AP1_3_01} 
\end{align}
This is equivalent to solving
\begin{align}
\u^{t+1} \hspace{-1pt} = \hspace{-1pt} 
\arg \min_{\u} 
\left( H_{1}(\u) + B_{D^{-1}}( \bm{\Phi} \u \| \tilde{\z}^{t} ) \right), 
\end{align}
where the Bregman divergence is used as a penalty term:
\begin{align}
B_{D^{-1}}( \bm{\Phi} \u \| \tilde{\z}^{t} ) = D^{-1}(\bm{\Phi} \u) - D^{-1}(\tilde{\z}^{t}) - \left< \nabla D^{-1}(\tilde{\z}^{t}), \bm{\Phi} \u - \tilde{\z}^{t} \right>. 
\end{align}
The update procedure in \eqref{eq:AP1_2_02} is simplified to
\begin{align}
\nabla D^{-1}(\tilde{\x}^{t+1}) \hspace{-1pt} &= \hspace{-1pt} \nabla D^{-1}(\tilde{\z}^{t}) - 2 \nabla D^{-1} (\bm{\Phi} \u^{t+1}),  \notag \\
\tilde{\x}^{t+1} \hspace{-1pt} &= \hspace{-1pt} \tilde{\z}^{t} - 2 \bm{\Phi} \u^{t+1}.
\end{align}
Therefore, the overall update procedure \eqref{eq:AP1_2_01}--\eqref{eq:AP1_2_04} is summarized by
\begin{align}
&\u^{t+1} \hspace{-1pt} = \hspace{-1pt} 
\arg \min_{\u} 
\left( H_{1}(\u) + B_{D^{-1}}( \bm{\Phi} \u \| \tilde{\z}^{t} ) \right), 
\label{eq:AP1_4_01} \\
&\tilde{\x}^{t+1} \hspace{-1pt} = \hspace{-1pt}
\tilde{\z}^{t} - 2 \bm{\Phi} \u^{t+1}, 
\label{eq:AP1_4_02} \\
&\v^{t+1} \hspace{-1pt} = \hspace{-1pt} 
\arg \min_{\v} 
\left( H_{2}(\v) + B_{D^{-1}}( -\v \| \tilde{\x}^{t+1} ) \right), 
\label{eq:AP1_4_03} 
\\
&\tilde{\z}^{t+1} \hspace{-1pt} = \hspace{-1pt}
\begin{cases}
\tilde{\x}^{t+1} + 2 \v^{t+1} 
& \textrm{(B-P-R splitting)} \\
\tilde{\z}^{t} - 2 \alpha \left( \bm{\Phi} \u^{t+1} -\v^{t+1}  \right)
& \textrm{(B-D-R splitting)} 
\end{cases}.
\label{eq:AP1_4_04}
\end{align}
The resulting algorithm is summarized in \textbf{Algorithm} \ref{alg:TVdenoising_BMOS}.

\begin{algorithm}[t]
   \caption{Bregman Peaceman-Rachford (B-P-R)/ Bregman Douglas-Rachford (B-D-R) splitting based TV denoising}
   \label{alg:TVdenoising_BMOS}
\begin{algorithmic}
\STATE{Initialization of $\tilde{\z}^{0}$} \\
\FOR{$t=0, \ldots, T \hspace{-2pt} - \hspace{-2pt}  1$}

\vspace{0pt}

\STATE
$\hspace{10pt} 
\u^{t+1} = \arg \min_{\u} \left( H_{1}(\u) 
+ B_{D^{-1}}( \bm{\Phi} \u \| \tilde{\z}^{t} )  \right), $ \\
$\hspace{10pt} 
\tilde{\x}^{t+1} = \tilde{\z}^{t} - 2 \bm{\Phi} \u^{t+1},$ \\
$\hspace{10pt} 
\v^{t+1} = \arg \min_{\v} \left( H_{2}(\v) + B_{D^{-1}}( -\v \| \tilde{\x}^{t+1}) \right), $ \\
$\hspace{10pt} 
\tilde{\z}^{t+1} = 
\begin{cases}
\tilde{\x}^{t+1} + 2 \v^{t+1} \hspace{50pt} 
\textrm{(B-P-R splitting)}  \\
\tilde{\z}^{t} - 2 \alpha ( \bm{\Phi} \u^{t+1} - \v^{t+1} ) 
\hspace{5pt} 
\textrm{(B-D-R splitting)} 
\end{cases}
$ \\

\vspace{0pt}
\ENDFOR
\end{algorithmic}
\end{algorithm}

Next, a Bregman divergence design is explained. 
As discussed in Sec. \ref{sec:2_4}, a practical choice of Bregman divergence is to use the second-order gradient (Hessian) of the cost function. 
The method can be useful even then the cost function is based on convex conjugation, as is the case here. 
When each cost is approximated by a strictly convex function using a second-order gradient, it is given by a quadratic form as: $\overline{H}_{1}(\u) \hspace{-2pt} = \hspace{-2pt} \frac{1}{2} \| \u \|_{2}^{2}$ and $\overline{H}_{2}(\v) \hspace{-2pt} = \hspace{-2pt} \frac{\mu \theta}{2} \| \v \|_{2}^{2}$. Note that the convex conjugate of a quadratic form is also quadratic. The Hessian of the convex conjugate of a quadratic is the inverse Hessian of the original quadratic \cite{bertsekas1978local}. Thus, we obtain $\overline{H}_{1}^{\star}( \bm{\Phi}^{\textrm{T}} \w ) \hspace{-1pt} = \hspace{-1pt} \frac{1}{2} \| \bm{\Phi}^{\textrm{T}} \w \|_{2}^{2}$, $\overline{H}_{2}^{\star}( -\w ) \hspace{-1pt} = \hspace{-1pt} \frac{1}{2 \mu \theta} \| \w \|_{2}^{2}$. Then, $\overline{H}^{\star} \hspace{-1pt} = \hspace{-1pt} \overline{H}^{\star}_{1} \hspace{-1pt} + \hspace{-1pt} \overline{H}^{\star}_{2}$ is given by
\begin{align}
\overline{H}^{\star} (\w) = \frac{1}{2} \left< \left( \frac{1}{\mu \theta} \I + \bm{\Phi} \bm{\Phi}^{\textrm{T}}  \right) \w, \w  \right>, 
\label{eq:HstrictConjugate_Design} 
\end{align}
where $\I$ is a unit matrix. Following the results in Sec. \ref{sec:2_4}, 
a reasonable choice of the Bregman divergence metric $D(\w)$ is of the form
\begin{align}
D(\w) &= \frac{1}{2} \left< \bm{\Psi} \w, \w \right>, 
\label{eq:AP1_Dinv_Design} 
\end{align}
where Newton \eqref{eq:bregman_Taylor_2GD}, AGD \eqref{eq:bregman_Taylor_AGD} and GD \eqref{eq:bregman_Taylor_GD} design methods can be used:
\begin{align}
\bm{\Psi}  = 
\begin{cases}
\M = \frac{1}{\mu \theta} \I + \bm{\Phi} \bm{\Phi}^{\textrm{T}}   &\textrm{(Newton)} \\
\L = \textrm{Diag} \left(  \frac{1}{\mu \theta} \I + \bm{\Phi} \bm{\Phi}^{\textrm{T}} \right) &\textrm{(AGD)} \\
\frac{1}{\kappa} \I &\textrm{(GD)} \\
\end{cases},
\label{eq:AP1_Psi_Design} 
\end{align}
where $\textrm{Diag}( \cdot )$ generates a diagonal matrix with the argument vector as diagonal. When we select $\bm{\Psi} \hspace{-1pt} = \hspace{-1pt} \frac{1}{\kappa} \I$ following GD, 
then \textbf{Algorithm} \ref{alg:TVdenoising_BMOS} reduces to a conventional Peaceman-Rachford and Douglas-Rachford splitting.


When the metric of Bregman divergence is given by \eqref{eq:AP1_Dinv_Design}, 
the $\u$-update procedure \eqref{eq:AP1_4_01} is given by an analytical form as
\begin{align}
\u^{t+1} = \left( \I + \bm{\Phi}^{\textrm{T}} \bm{\Psi}^{-1} \bm{\Phi} \right)^{-1} 
\left(  \s + \bm{\Psi}^{-1} \z^{t} \right). 
\end{align}
Meanwhile, the $\v$-update procedure \eqref{eq:AP1_4_03} implies that the subdifferential of the cost includes zero as
\begin{align}
\bm{0} \in \mu \left( \theta \v + \partial \| \v \|_{1} \right) + \bm{\Psi}^{-1} \v + \bm{\Psi}^{-1} \tilde{\x}^{t+1}. 
\end{align}
For the subdifferential of $L_{1}$ norm $\bm{\xi} \hspace{-2pt} = \hspace{-2pt} \partial \| \v \|_{1} $, 
the $i$-th element of is calculated by 
\begin{align}
\xi_{i}^{t+1} = 
\begin{cases}
1 & \left( \left[ \bm{\Psi}^{-1} \tilde{\x}^{t+1} \right]_{i} > \mu \right) \\
-\left[ \frac{1}{\mu} \bm{\Psi}^{-1} \tilde{\x}^{t+1} \right]_{i} & \left( -\mu \leq \left[ \bm{\Psi}^{-1} \tilde{\x}^{t+1} \right]_{i} \leq \mu \right)  \\
-1 & \left( \left[ \bm{\Psi}^{-1} \tilde{\x}^{t+1} \right]_{i} < - \mu \right)
\end{cases}.
\end{align}
Thus, the $\v$-update procedure is given by
\begin{align}
v_{i}^{t+1} = 
\begin{cases}
0 & \left( -\mu \leq \left[ \bm{\Psi}^{-1} \tilde{\x}^{t+1} \right]_{i} \leq \mu \right)
\\
\left[ - \left( \mu \theta \I + \bm{\Psi}^{-1} \right)^{-1} 
\left( \bm{\Psi}^{-1} \tilde{\x}^{t+1} + \mu \bm{\xi}^{t+1} \right) \right]_{i} & \textrm{(otherwise)}
\end{cases}. 
\end{align}

\begin{table}[t]
\centering
\caption{Parameter Settings.}
\vspace{-5pt}
\label{tbl:param}
\begin{tabular}{|l|l|l|}
\hline
Parameter & Algorithm ($B_{D}$ design) & Value \\\hline
Elastic net normalization coefficient, $\mu$ & all & 2.0 \\ \hline
Squared $L_{2}$ normalization coefficient, $\theta$ & all & 1.0 \\ \hline
Step-size used in conventional methods, $\kappa$ & P-R/D-R & 0.01 \\ \hline
Averaging coefficient, $\alpha$ & B-D-R (Newton/AGD), D-R & 0.5 \\ \hline
\end{tabular}
\vspace{-3pt}
\end{table}

\begin{figure}[t]
\vskip 0.1in
\begin{center}
\centerline{\includegraphics[width=140mm]{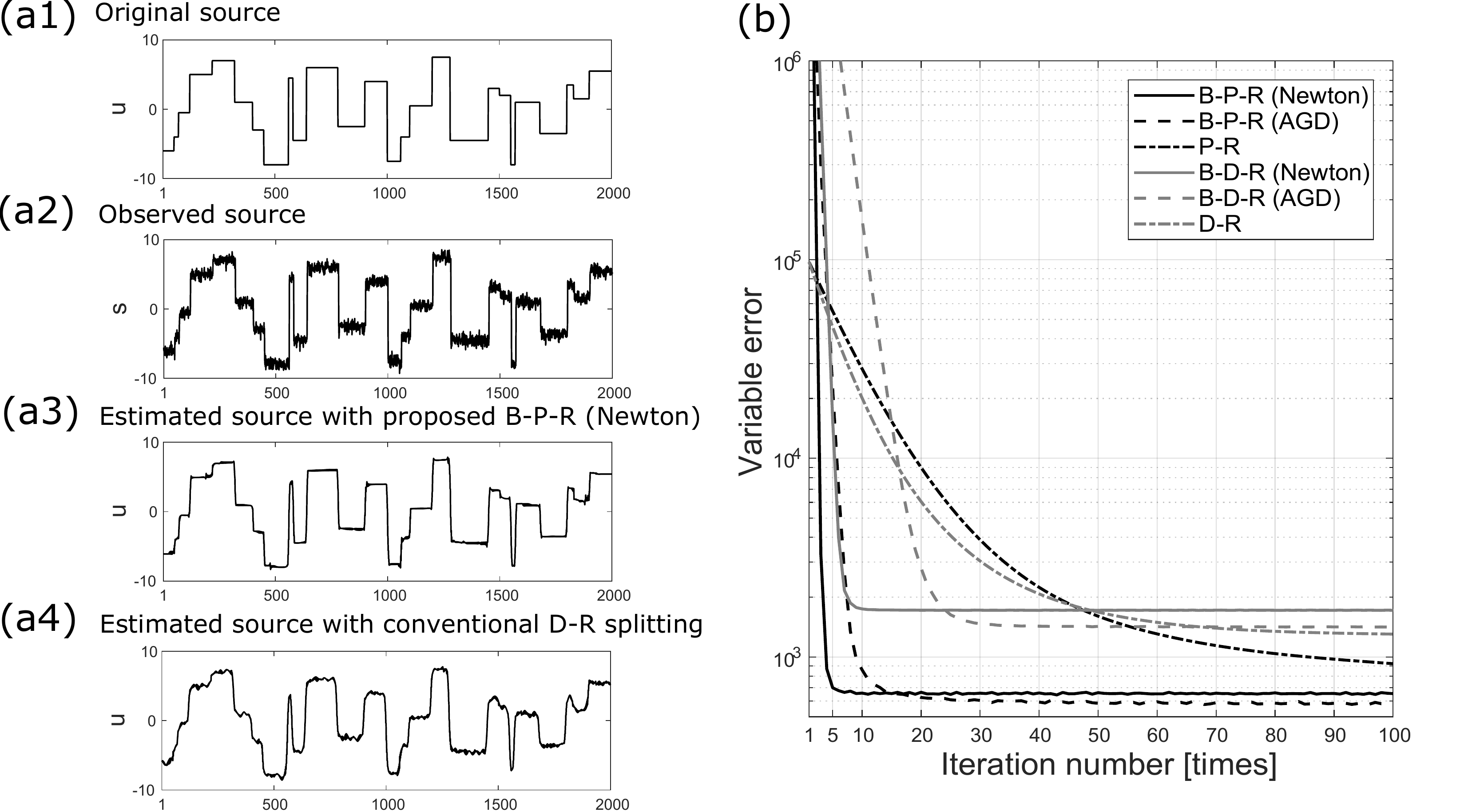}}
\vskip -0.1in
\caption{Results of numerical experiments, 
(a1) original source $\u_{\textrm{GT}}$, 
(a2) observed source $\s \hspace{-1pt} = \hspace{-1pt} \u_{\textrm{GT}} \hspace{-1pt} + \hspace{-1pt} \e$, 
(a3) estimated source $\u$ with proposed Bregman Peaceman-Rachford splitting with Newton method, 
(a4) estimated source $\u$ with conventional Douglas-Rachford splitting, 
(b) convergence rate curves.
}
\label{fig:res_ap1}
\end{center}
\vskip -0.1in
\end{figure}

\subsection{Numerical Experiments}
\label{sec:3_3}

The convergence rates of several B-MOS algorithms were compared with conventional MOS methods. 
Bregman Peaceman-Rachford (B-P-R) and Bregman Douglas-Rachford (B-D-R) splitting in the form summarized in \textbf{Algorithm} \ref{alg:TVdenoising_BMOS} were applied. To those algorithms, three kinds of Bregman divergence metric design \eqref{eq:AP1_Dinv_Design} and \eqref{eq:AP1_Psi_Design} are available. In total six algorithm forms are obtained from their combination, 
of which two methods, which use Euclidean distance as a Bregman divergence metric design, 
are conventional Peaceman-Rachford (P-R) and Douglas-Rachford (D-R) splitting.

As an example, we generate a source that is piecewise constant as shown in Fig. \ref{fig:res_ap1} (a1). 
Given the ground truth vector $\u_{\textrm{GT}}$ whose dimension is $m \hspace{-2pt} = \hspace{-2pt} 2000$, 
the observed source $\s$ is obtained by $\s \hspace{-2pt} = \hspace{-2pt} \u_{\textrm{GT}} + \e$ 
where noise $\e$ is drawn from a normal distribution $\mathrm{Norm}(0, 0.5)$ and it is shown in Fig. \ref{fig:res_ap1} (a2). 
Several parameters used in this experiments are summarized in Table \ref{tbl:param}. 
As an evaluation measure, the variable error $E^{t}$ defined by 
the squared error between the estimated variable $\u^{t}$ and its ground truth $\u_{\textrm{GT}}$ was used as
\begin{align}
E^{t} = \frac{1}{2} \| \u_{\textrm{GT}} - \u^{t} \|_{2}^{2}. 
\label{eq:ap1_err}
\end{align}

The resulting variables with proposed B-P-R with Newton and conventional D-R splitting are shown in Fig. \ref{fig:res_ap1} (a3) and Fig. \ref{fig:res_ap1} (a4), respectively. 
An estimated variable close to the ground truth was obtained. 
Figure \ref{fig:res_ap1} (b) shows the relationships between the six methods and the variable error $E^{t}$. 
The experimental results show that B-P-R with Newton had the fastest convergence rates followed by B-P-R with AGD and B-D-R with Newton. The convergence rates with conventional P-R and D-R splitting were slow for this task. 
A major advantage of the new method is that we do not have to set a learning rate.

\section{CONCLUSION}
\label{sec:5}

We considered the use of operator splitting to find the infimum of
$G(\w) \hspace{-2pt} = \hspace{-2pt} G_1(\w) \hspace{-2pt} + \hspace{-2pt} G_2(\w)$, where $G_1$ and $G_2$ are convex, closed proper functions.
We proposed a generalization of monotone operator splitting (MOS) based on Bregman
divergence (B-MOS). The convergence rates of the generalized approach depend on the choice
for the Bregman divergence. 
We found that fast a convergence rate can be achieved by
designing the function $D$ that characterizes the Bregman divergence $B_D(\w\| \z)$ such
that $\nabla D^{-1} \partial G_i$ is near the identity operator. 
Since the cost function is composed of two CCP functions, $D$ is matched to each CCP
function for each update.
A major advantage of the new method is it eliminates the need to  carefully set learning rates. 
The outcomes of our
numerical experiments, in which the B-MOS  solvers were applied to a constrained
optimization problem, revealed that B-MOS solvers  can significantly improve the
convergence rate in practical optimization problems.

\appendix
\section{Attributes of $D$-Resolvent Operator, $D$-Cayley Operator and $D$-Forward Step}
\label{sec:appendix}

$D$-resolvent operator, the $D$-Cayley operator and the $D$-forward step. 
To this purpose, we first model that how the property of $G_{i}$ will be modified by applying $\nabla D^{-1}$ to $\partial G_{i}$.

We assume that $\partial G_{i}$ satisfies
\begin{align}
\gamma_{\textrm{LB}, i}
\| \w \hspace{-1pt} - \hspace{-1pt} \z  \|_{2}
\leq 
\| \partial G_{i}(\w) - \partial G_{i}(\z) \|_{2}
\leq 
\gamma_{\textrm{UB}, i} \| \w \hspace{-1pt} - \hspace{-1pt} \z \|_{2},
\label{eq:G_2bounds}
\end{align}
for any two different points $\w \in \mathrm{dom}(G_i)$ and $\z \in \mathrm{dom}(G_i)$, and where $0 \leq
\gamma_{\textrm{LB}, i} \leq \gamma_{\textrm{UB}, i} < +\infty $.  
Applying $\nabla D^{-1}$ to $\partial G_{i}$ modifies the property of $G_{i}$ to
\begin{align}
\sigma_{\mathrm{LB}, i} \| \w \hspace{-2pt} - \hspace{-2pt} \z \|_{2}
\leq 
\| \nabla D^{-1}  \partial G_{i}(\w) - \nabla D^{-1} \partial G_{i}(\z)  \|_{2} 
\leq 
\sigma_{\mathrm{UB}, i} \| \w \hspace{-1pt} - \hspace{-1pt} \z \|_{2},
\label{eq:G_4bounds}
\end{align}
where $0 \leq \sigma_{\mathrm{LB}, i} \leq \sigma_{\mathrm{UB}, i} < +\infty$. 
Note that $\nabla D^{-1} \partial G_{i}$ is assumed to be Lipschitz continuous,
but is not necessarily strongly convex. 
In this Appendix 
we find the optimal pair of $\{ \sigma_{\mathrm{UB}, i}, \sigma_{\mathrm{LB}, i} \}$
for fast convergence using B-MOS algorithms.

We can now derive the Lipschitz continuity of the $D$-resolvent operator, $D$-Cayley
operator and $D$-forward step (with assumption) using $\{ \sigma_{\mathrm{UB}, i}, \sigma_{\mathrm{LB}, i} \}$.

\begin{theorem}{Nonexpansive property of $D$-resolvent operator} \label{th:R1}

Let $\nabla D^{-1} \partial G_{i}$ be Lipschitz continuous on $\mathrm{dom} (G_i)$, i.e., 
$\{ \sigma_{\mathrm{UB}, i}, \sigma_{\mathrm{LB}, i} \}$ that satisfy
$0 \hspace{-2pt} \leq  \hspace{-2pt} \sigma_{\mathrm{LB}, i} \hspace{-2pt} \leq \hspace{-2pt} \sigma_{\mathrm{UB}, i} \hspace{-2pt} < \hspace{-2pt} + \infty$ in \eqref{eq:G_4bounds} exist. 
Then, the contractive ratio for the input/output pairs on the $D$-resolvent operator $R_{i}$ is given by
\begin{align}
\frac{1}{1 \hspace{-1pt} + \hspace{-1pt} \sigma_{\mathrm{UB}, i} } \hspace{-1pt}
\parallel \hspace{-1pt} \z^{t} \hspace{-1pt} - \hspace{-1pt} \z^{t-1} \hspace{-1pt} \parallel_{2}
&\hspace{-1pt} \leq 
\parallel \hspace{-1pt} R_{i}( \z^{t} ) \hspace{-1pt} - \hspace{-1pt} R_{i}( \z^{t-1} )  \hspace{-1pt} \parallel_{2} 
\leq \hspace{-1pt}
\frac{1}{1 \hspace{-1pt} + \hspace{-1pt} \sigma_{\mathrm{LB}, i} }  \hspace{-1pt}
\parallel \hspace{-1pt} \z^{t} \hspace{-1pt} - \hspace{-1pt} \z^{t-1} \hspace{-1pt} \parallel_{2}. 
\label{eq:resolvent_converge_rate}
\end{align}
When $\nabla D^{-1} \partial G_{i}$ is strongly monotone, i.e., $\sigma_{\mathrm{LB}, i} > 0$, $R_{i}$ is a contractive operator. 
Otherwise, $R_{i}$ is a nonexpansive operator.

\end{theorem}

\begin{proof}
%
The input/output pairs for the $D$-resolvent operator $R_{i}=(1 + \nabla D^{-1} \partial G_{i} )^{-1}$ are $\w^{t} \hspace{-1pt} =
\hspace{-1pt} R_{i}(\z^{t-1})$, $\w^{t+1} \hspace{-1pt} = \hspace{-1pt} R_{i}(\z^{t})$. They are reformulated as
\begin{align}
\hspace{-2pt} ( I + \nabla D^{-1} \partial G_{i} ) (\w^{t}) \hspace{-1pt} = \hspace{-1pt} \z^{t-1}, 
\hspace{5pt}
(I + \nabla D^{-1} \partial G_{i} ) (\w^{t+1}) \hspace{-1pt} = \hspace{-1pt} \z^{t}. \notag
\end{align}
By subtracting these, we obtain
\begin{align}
&\hspace{-2pt} \left( I \hspace{-1pt} + \hspace{-1pt} \nabla D^{-1} \partial G_{i} \right) \hspace{-1pt} (\w^{t+1}) \hspace{-1pt} 
- \hspace{-1pt} \left( I \hspace{-1pt} + \hspace{-1pt} \nabla D^{-1} \partial G_{i} \right) \hspace{-1pt} (\w^{t}) \hspace{-1pt} = \hspace{-1pt} \z^{t} \hspace{-1pt} - \hspace{-1pt} \z^{t-1}. 
\label{eq:resolvent_tmp01}
\end{align}
Since $(I \hspace{-1pt} + \hspace{-1pt} \nabla D^{-1} \partial G_{i})$ is strongly monotone with 
$(1 \hspace{-1pt} + \hspace{-1pt} \sigma_{\mathrm{LB}, i} )$, 
its inverse operator $(I \hspace{-1pt} + \hspace{-1pt} \nabla D^{-1} \partial G_{i})^{-1} = \hspace{-1pt} R_{i}$ is Lipschitz continuous 
with $(1 \hspace{-1pt} + \hspace{-1pt} \sigma_{\mathrm{LB}, i} )^{-1}$, e.g., \cite{ryu2016primer}. 
Hence, the upper bound in (\ref{eq:resolvent_converge_rate}) is proven. 
Since $\sigma_{\mathrm{LB}, i} \hspace{-1pt} \geq \hspace{-1pt} 0$, 
this shows the nonexpansive property of $D$-resolvent operator and this fact was first proven in \cite{bauschke2003bregman}. 
By taking the norm of (\ref{eq:resolvent_tmp01}), we obtain
\begin{align}
&\parallel \hspace{-1pt} \w^{t+1} \hspace{-1pt} - \hspace{-1pt} \w^{t} \hspace{-1pt} \parallel_{2}
\hspace{-1pt} + \hspace{-1pt} \parallel \hspace{-1pt} \nabla D^{-1} \partial G_{i} (\w^{t+1}) \hspace{-1pt} 
- \hspace{-1pt} \nabla D^{-1} \partial G_{i} (\w^{t}) \hspace{-1pt} \parallel_{2} 
\geq \parallel \hspace{-1pt} \z^{t} \hspace{-1pt} - \hspace{-1pt} \z^{t-1} \hspace{-1pt} \parallel_{2}. 
\end{align}
Since $\nabla D^{-1} \partial G_{i}$ is assumed to be Lipschitz continuous as in \eqref{eq:G_4bounds}, 
the lower bound in (\ref{eq:resolvent_converge_rate}) is obtained. 
\end{proof}

\begin{theorem}{Nonexpansive property of $D$-Cayley operator} \label{th:C1}

Let $\nabla D^{-1} \partial G_{i}$ be Lipschitz continuous on $\mathrm{dom} (G_i)$, i.e., 
$\{ \sigma_{\mathrm{UB}, i}, \sigma_{\mathrm{LB}, i} \}$ that satisfy
$0 \hspace{-2pt} \leq \hspace{-2pt} \sigma_{\mathrm{LB}, i} \hspace{-2pt} \leq \hspace{-2pt} \sigma_{\mathrm{UB}, i} \hspace{-2pt} < \hspace{-2pt} + \infty$ in \eqref{eq:G_4bounds} exist. 
Then, the contractive ratio for the input/output pairs on the $D$-Cayley operator $C_{i}$ satisfies
\begin{align}
&\parallel \hspace{-1pt} C_{i}(\z^{t}) - C_{i}(\z^{t-1}) \hspace{-1pt} \parallel_{2}
\leq \eta_{i} 
\parallel \hspace{-1pt} \z^{t} - \z^{t-1} \hspace{-1pt} \parallel_{2},
\label{eq:Cayley_convergence_rate}
\end{align}
where $\eta_{i} \hspace{2pt} (0 \leq \eta_{i} \leq 1)$ is defined by 
\begin{align}
\eta_{i} = \sqrt{ 1 - \frac{4 \sigma_{\mathrm{LB}, i}}{ (1 + \sigma_{\mathrm{UB}, i} )^{2} }}.
\label{eq:eta1}
\end{align}
When $\nabla D^{-1} \partial G_{i}$ is strongly monotone, i.e., $\sigma_{\mathrm{LB}, i} > 0$, $C_{i}$ is a contractive operator. 
Otherwise,  $C_{i}$ is a nonexpansive operator. 
\end{theorem}

\begin{proof}

When we have $\w^{t} \hspace{-2pt} = \hspace{-2pt} R_{i} (\z^{t-1})$ and $\w^{t+1}
\hspace{-2pt} = \hspace{-2pt} R_{i} (\z^{t})$ of Theorem \ref{th:R1} holds, 
we obtain the following relationship by multiplying
$(\w^{t+1} \hspace{-2pt}  - \hspace{-2pt}  \w^{t})^{\textrm{T}}$ with
(\ref{eq:resolvent_tmp01}) as
\begin{align}
& \parallel \hspace{-3pt} \w^{t+1} \hspace{-3pt} - \hspace{-2pt} \w^{t} \hspace{-2pt} \parallel_{2}^{2} 
\hspace{-2pt} + \hspace{-2pt} \left< \w^{t+1} \hspace{-2pt} - \hspace{-1pt} \w^{t}, \nabla \hspace{-1pt} D^{-1} \partial G_{i} (\w^{t+1}) \hspace{-2pt} - \hspace{-2pt}
\nabla \hspace{-1pt} D^{-1} \partial G_{i} (\w^{t}) \hspace{-1pt} \right> 
\hspace{-2pt} = \hspace{-2pt} \left< \w^{t+1} \hspace{-2pt} - \hspace{-1pt}  \w^{t}, \z^{t} \hspace{-2pt} - \hspace{-2pt} \z^{t-1} \right>\hspace{-2pt}. \notag
\hspace{-5pt}
\end{align}
From the lower bound in \eqref{eq:G_4bounds}, we obtain
\begin{align}
\hspace{-4pt}
\left( 1 + \sigma_{\mathrm{LB}, i} \right)  \parallel \hspace{-2pt} \w^{t+1} \hspace{-2pt} - \hspace{-2pt} \w^{t} \hspace{-2pt} \parallel_{2}^{2} \hspace{2pt}
 \leq  \left< \w^{t+1} \hspace{-2pt} - \hspace{-2pt} \w^{t}, \z^{t} \hspace{-2pt} - \hspace{-2pt} \z^{t-1} \right>. 
\label{eq:Cayley_tmp01}
\end{align}
By taking the squared norm for the $D$-Cayley input/output pairs 
$\x^{t} \hspace{-2pt} = \hspace{-2pt} C_{i} (\z^{t-1})$, $\x^{t+1} \hspace{-2pt} = \hspace{-2pt} C_{i} (\z^{t})$, we obtain
\begin{subequations}
\begin{align}
\hspace{-10pt}  \parallel \hspace{-1pt} \x^{t+1} - \x^{t} \hspace{-1pt} \parallel_{2}^{2} 
&= \parallel \hspace{-1pt} 2( \w^{t+1} - \w^{t} ) - ( \z^{t} - \z^{t-1}) \hspace{-1pt} \parallel_{2}^{2} \notag \\
&= 4 \hspace{-1pt} \parallel \hspace{-3pt} \w^{t+1} \hspace{-2pt} - \hspace{-2pt} \w^{t} \hspace{-3pt} \parallel_{2}^{2} \hspace{-2pt}
-  4 \hspace{0pt} \left< \w^{t+1} \hspace{-3pt} - \hspace{-2pt} \w^{t}, \z^{t} \hspace{-2pt} - \hspace{-2pt} \z^{t-1} \right> \hspace{-2pt}
+ \hspace{-2pt} \parallel \hspace{-3pt} \z^{t} \hspace{-1pt} - \hspace{-1pt} \z^{t-1} \hspace{-3pt} \parallel_{2}^{2}  \hspace{-4pt}
\label{eq:Cayley_nonexpansive_tmp1} \\
&\leq \parallel \hspace{-1pt} \z^{t} - \z^{t-1} \hspace{-1pt} \parallel_{2}^{2}, 
\label{eq:Cayley_nonexpansive}
\end{align}
\end{subequations}
where \eqref{eq:Cayley_tmp01} is used for reforming \eqref{eq:Cayley_nonexpansive_tmp1} into \eqref{eq:Cayley_nonexpansive}, 
and this proves the nonexpansive property of $C_{i}$. 
Combining (\ref{eq:Cayley_tmp01}) and (\ref{eq:Cayley_nonexpansive_tmp1}) results in
\begin{align}
\parallel \hspace{-2pt} \x^{t+1} - \x^{t} \hspace{-2pt} \parallel_{2}^{2} 
\leq \parallel \hspace{-2pt} \z^{t} - \z^{t-1} \hspace{-2pt} \parallel_{2}^{2}
 -  4 \sigma_{\mathrm{LB}, i}
\parallel \hspace{-2pt} \w^{t+1} - \w^{t} \hspace{-2pt} \parallel_{2}^{2}. \notag
\end{align}
With the lower bound of (\ref{eq:resolvent_converge_rate}), we obtain
\begin{align}
\parallel \hspace{-2pt} \x^{t+1} - \x^{t} \hspace{-2pt} \parallel_{2}^{2} 
&\leq 
\left( 1 - \frac{4 \sigma_{\mathrm{LB}, i} }{ (1 + \sigma_{\mathrm{UB}, i} )^{2} } \right) \parallel \hspace{-2pt} \z^{t} - \z^{t-1} \hspace{-2pt} \parallel_{2}^{2}. \notag
\end{align}
Therefore, we obtain (\ref{eq:Cayley_convergence_rate}). 
\end{proof}

Next, we find the optimal values for $\{ \sigma_{\mathrm{UB}, i}, \sigma_{\mathrm{LB}, i} \}$ when $D$-Cayley operator is used. 
Let us optimize $\sigma_{\mathrm{LB}, i}$ given $\sigma_{\mathrm{UB}, i} \geq 0$. 
It is clear that this is the case for $\sigma_{\mathrm{LB}, i} \hspace{-1pt} = \hspace{-1pt} \min( \sigma_{\mathrm{UB}, i}, \frac{1}{4}(1 \hspace{-1pt} + \hspace{-1pt} \sigma_{\mathrm{UB}, i} )^2)$. 
This means that 
$\sigma_{\mathrm{LB}, i} \hspace{-1pt} = \hspace{-1pt} \sigma_{\mathrm{UB}, i} = \frac{1}{4}(1 \hspace{-1pt} + \hspace{-1pt} \sigma_{\mathrm{UB}, i} )^2$ only if $\sigma_{\mathrm{UB}, i} \hspace{-1pt} = \hspace{-1pt} 1$ and the contraction factor $\eta_{i}$ is then equal to 0. 
For $0 \hspace{-1pt} \leq \hspace{-1pt} \sigma_{\mathrm{UB}, i} \hspace{-1pt} <
\hspace{-1pt} 1$ or $\sigma_{\mathrm{UB}, i} \hspace{-1pt} > \hspace{-1pt} 1$,  
the optimal contraction factor results when $\sigma_{\mathrm{LB}, i}
\hspace{-1pt} = \hspace{-1pt} \sigma_{\mathrm{UB}, i}$. Thus, the contraction factor
$\eta_i$ satisfies
\begin{align}
0 \leq \sqrt{1 - \frac{4 \sigma_{\mathrm{UB}, i} }{(1+\sigma_{\mathrm{UB}, i}
  )^{2}}} \leq \eta_i \leq 1.
\end{align}
We conclude that optimal contraction for $D$-Cayley operator is obtained when 
\begin{align}
\sigma_{\mathrm{LB}, i} = 1, \hspace{5pt} \sigma_{\mathrm{UB}, i} = 1.
\label{optimalsigmapair_DCayley} 
\end{align}
Then the contractive ratio is obtained as $\eta_{i} = 0$. 
Moreover, for a given $\sigma_{\mathrm{UB}, i} $ it is optimal to minimize the dynamic 
range to $\sigma_{\mathrm{UB}, i} / \sigma_{\mathrm{LB}, i} = 1$.

\begin{theorem}{Lipschitz continuity of $D$-forward step} \label{th:Q1}

Let $\nabla D^{-1} \partial G_{i}$ be Lipschitz continuous on $\mathrm{dom} (G_i)$, i.e., 
$\{ \sigma_{\mathrm{UB}, i}, \sigma_{\mathrm{LB}, i} \}$ that satisfy
$0 \hspace{-1pt} \leq  \hspace{-1pt} \sigma_{\mathrm{LB}, i} \hspace{-1pt} \leq \hspace{-1pt} \sigma_{\mathrm{UB}, i} \hspace{-1pt} < \hspace{-1pt} + \infty$ in \eqref{eq:G_4bounds} exist. 
Then, the input/output pairs on the $D$-forward step $F_{i}$ satisfy:
\begin{align}
\parallel \hspace{-1pt} F_{i}( \z^{t} ) \hspace{-1pt} - \hspace{-1pt} F_{i}( \z^{t-1} )  \hspace{-1pt} \parallel_{2} 
\leq \hspace{-1pt}
\nu_{i} \parallel \hspace{-1pt} \z^{t} \hspace{-1pt} - \hspace{-1pt} \z^{t-1} \hspace{-1pt} \parallel_{2}, 
\label{eq:gradient_converge_rate}
\end{align}
where $\nu_{i} \geq 0$ is given by
\begin{align}
\nu_{i} = \sqrt{ 1 - 2 \sigma_{\textrm{LB}, i} + \sigma_{\textrm{UB}, i}^{2} }. 
\label{eq:nu}
\end{align}

\end{theorem}

\begin{proof}

Consider $\z^t \hspace{-2pt} \in \hspace{-2pt} \mathrm{dom}(G_i)$ and $\z^{t+1} \in \mathrm{dom}(G_i)$  and the
$D$-forward step $F_{i} \hspace{-2pt} = \hspace{-2pt} (I  \hspace{-2pt} - \hspace{-2pt} \nabla D^{-1} \partial G_{i})$. 
Let $\w^{t} \hspace{-1pt} =
\hspace{-1pt} F_{i}(\z^{t-1})$, $\w^{t+1} \hspace{-1pt} = \hspace{-1pt} F_{i}(\z^{t})$.
%
The $L_{2}$ norm of the difference $\w^{t+1} - \w^{t}$ is then bounded by
%
\begin{align}
\parallel \hspace{-1pt} \w^{t+1} - \w^{t}  \hspace{-1pt} \parallel_{2}^{2} &=
\parallel \hspace{-1pt} 
(I \hspace{-1pt} - \hspace{-1pt} \nabla D^{-1} \partial G_{i}) (\z^{t})
- (I \hspace{-1pt} - \hspace{-1pt} \nabla D^{-1} \partial G_{i}) (\z^{t-1})
\hspace{-1pt} \parallel_{2}^{2} \notag \\
&=
\parallel \hspace{-2pt} 
\z^{t} - \z^{t-1}
- (\nabla D^{-1} \partial G_{i} (\z^{t}) -  \nabla D^{-1} \partial G_{i} (\z^{t-1}) )
\hspace{-2pt} \parallel_{2}^{2}
\notag \\
&= 
\parallel \hspace{-2pt} 
\z^{t} - \z^{t-1}
\parallel_{2}^{2} 
- 2 \left< \nabla D^{-1} \partial G_{i} (\z^{t}) -  \nabla D^{-1} \partial G_{i} (\z^{t-1}), \z^{t} - \z^{t-1} \right> \notag \\
& \hspace{120pt} + \parallel \nabla D^{-1} \partial G_{i} (\z^{t}) -  \nabla D^{-1} \partial G_{i} (\z^{t-1}) \hspace{-2pt} \parallel_{2}^{2}
\notag \\
&\leq  
(1 - 2 \sigma_{\textrm{LB}, i} + \sigma_{\textrm{UB}, i}^{2} )
\parallel \hspace{-2pt} 
\z^{t} - \z^{t-1}
\parallel_{2}^{2}
\label{Lipchitz_I_DinvpG2}. 
\end{align}
\end{proof}

\begin{figure}[t]
\vskip 0.1in
\begin{center}
\centerline{\includegraphics[width=90mm]{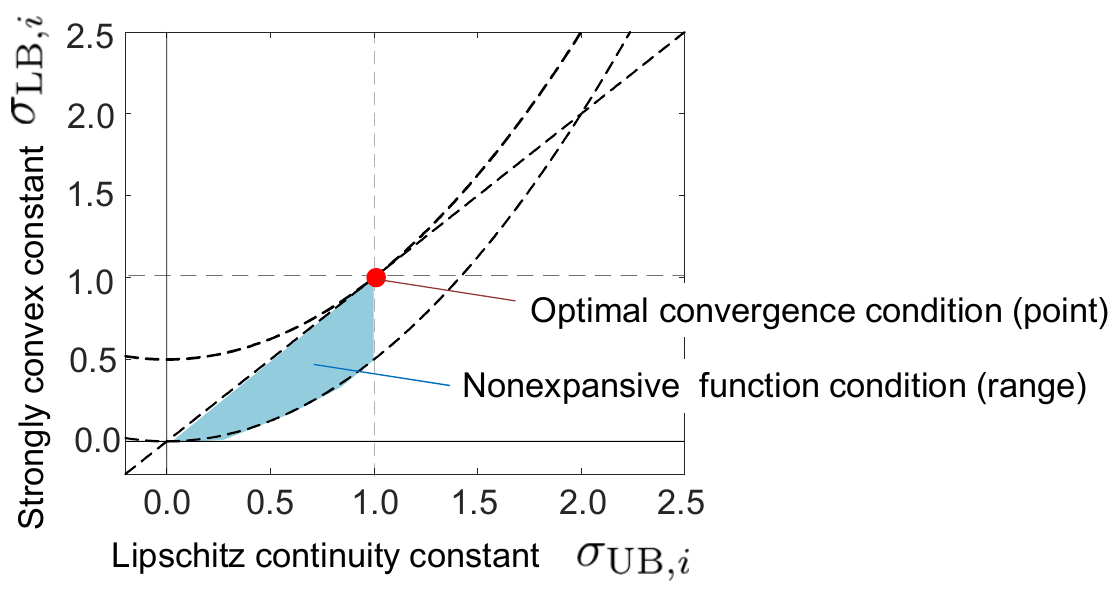}}
\vskip -0.12in
\caption{Requirement to make $D$-forward step nonexpansive operator. }
\label{fig:nonexpansivecondition}
\end{center}
\vskip -0.2in
\end{figure}

We now study the value range of $\{ \sigma_{\textrm{UB}, i}, \hspace{-1pt} \sigma_{\textrm{LB}, i} \}$ that makes the $D$-forward step a function (one-to-one mapping). For the input/output pairs for the $D$-forward step, we can write
\begin{align}
\parallel \hspace{-1pt} \w^{t+1} - \w^{t}  \hspace{-1pt} \parallel_{2} &=
\parallel \hspace{-1pt} 
(I \hspace{-1pt} - \hspace{-1pt} \nabla D^{-1} \partial G_{i}) (\z^{t})
- (I \hspace{-1pt} - \hspace{-1pt} \nabla D^{-1} \partial G_{i}) (\z^{t-1})
\hspace{-1pt} \parallel_{2} \notag \\
&=
\parallel \hspace{-2pt} 
\z^{t} - \z^{t-1}
- (\nabla D^{-1} \partial G_{i} (\z^{t}) -  \nabla D^{-1} \partial G_{i} (\z^{t-1}) )
\hspace{-2pt} \parallel_{2}.
\label{eq:ap1_tmp000}
\end{align}
When $\nabla D^{-1} \partial G_{i}$ is assumed to be nonexpansive operator, i.e., it satisfies $0 \leq \sigma_{\textrm{UB}, i} \leq 1$ in \eqref{eq:G_4bounds}, \eqref{eq:ap1_tmp000} is reformulated by
\begin{align}
\parallel \hspace{-1pt} \w^{t+1} - \w^{t}  \hspace{-1pt} \parallel_{2}
&\geq 
\parallel \hspace{-2pt} 
\z^{t} - \z^{t-1}
\hspace{-2pt} \parallel_{2}
- 
\parallel \hspace{-2pt} 
\nabla D^{-1} \partial G_{i} (\z^{t}) -  \nabla D^{-1} \partial G_{i} (\z^{t-1})
\hspace{-2pt} \parallel_{2}
\notag \\
&\geq 
(1 - \sigma_{\textrm{UB}, i} )
\parallel \hspace{-2pt} 
\z^{t} - \z^{t-1}
\hspace{-2pt} \parallel_{2}.
\end{align}
While satisfying $0 \hspace{-2pt} \leq \hspace{-2pt} (1 - \sigma_{\textrm{UB}, i} ) \hspace{-2pt} \leq \hspace{-2pt} 1$, i.e., $0 \leq \sigma_{\textrm{UB}, i} \leq 1$, $D$-forward step is a function.

Next, the condition to make the $D$-forward step a nonexpansive function is investigated. 
The nonexpansive condition $(0 \hspace{-1pt} \leq \hspace{-1pt} \nu_{i} \hspace{-1pt} \leq \hspace{-1pt} 1)$ is equivalent to 
$\frac{1}{2} \sigma_{\mathrm{UB}, i}^{2} \hspace{-1pt} \leq \hspace{-1pt} \sigma_{\textrm{LB}, i} \hspace{-1pt} \leq \hspace{-1pt} \frac{1}{2} 
( \sigma_{\mathrm{UB}, i}^{2} + 1 )$. 
Integrated with the condition $0 \hspace{-1pt} \leq \hspace{-1pt} \sigma_{\textrm{LB}, i} \hspace{-1pt} \leq \hspace{-1pt} \sigma_{\textrm{UB}, i} \hspace{-1pt} \leq \hspace{-1pt} 1$, the requirement to make the $D$-forward step a nonexpansive function
is summarized in Fig. \ref{fig:nonexpansivecondition}. The contraction factor of the $D$-forward step $\nu_{i} = 0$ corresponds to the unique solution
\begin{align}
\sigma_{\mathrm{LB}, i} = 1, \hspace{5pt} \sigma_{\mathrm{UB}, i} = 1. 
\label{optimalsigmapair_Dforward} 
\end{align}
This conclusion for the $\{ \sigma_{\textrm{UB}, i}, \hspace{-1pt} \sigma_{\textrm{LB}, i} \}$-optimization for $D$-forward step is in the nonexpansive function condition as shown in in Fig. \ref{fig:nonexpansivecondition} and it is equivalent to that for $D$-Cayley operator as in \eqref{optimalsigmapair_DCayley}.

\section{Convergence Rates on B-MOS Algorithms}
\label{sec:appendixB}

In Appendix \ref{sec:appendixB}, the convergence rates of B-MOS algorithms are investigated. 
Since these algorithms are based on the $D$-resolvent operator, $D$-Cayley operator and $D$-forward step, 
B-MOS convergence rates depend strongly on their Lipschitz continuity property explained in Appendix \ref{sec:appendix}. 
As discussed in Appendix \ref{sec:appendix}, 
$\nabla D^{-1} \partial G_{i}$ is assumed to be Lipschitz continuous, 
i.e., a pair $\{ \sigma_{\mathrm{UB}, i}, \sigma_{\mathrm{LB}, i} \}$ exists for
\eqref{eq:G_4bounds} , such that $0 \hspace{-2pt} \leq \hspace{-2pt} \sigma_{\mathrm{LB}, i} \hspace{-2pt} \leq \hspace{-2pt}
\sigma_{\mathrm{UB}, i} \hspace{-2pt} < \hspace{-2pt} + \infty$.

We first derive the convergence rate of Bregman Peaceman-Rachford splitting \eqref{eq:PRsplit}. 
From Theorem \ref{th:C1}, 
we found that the contractive ratio of the $D$-Cayley operator $C_{i}$ is provided by $\eta_{i}$ as in \eqref{eq:eta1}. 
For subsequent input/output pairs of Bregman Peaceman-Rachford splitting, 
$\z^{t+1}
\hspace{-1pt} = \hspace{-1pt} C_{2} C_{1} ( \z^{t} ), \z^{t} \hspace{-1pt} = \hspace{-1pt}
C_{2} C_{1} ( \z^{t-1} )$,  it follows from Theorem \eqref{th:C1} that the contractive ratio can be bounded by
\begin{align}
\parallel \hspace{-1pt} \z^{t+1} \hspace{-1pt} - \hspace{-1pt} \z^{t} \hspace{-1pt} \parallel_{2}
\leq \eta_{1} \eta_{2} \hspace{-1pt}
\parallel \hspace{-1pt} \z^{t} \hspace{-1pt} - \hspace{-1pt} \z^{t-1} \hspace{-1pt} \parallel_{2}.
\label{eq:contractiverate_PR} 
\end{align}
The difference between variable $\z^{t}$ and its fixed point $\z^{\ast}$ 
is represented by 
\begin{align}
\parallel \hspace{-1pt} \z^{t} \hspace{-1pt} - \hspace{-1pt} \z^{\ast} \hspace{-1pt} \parallel_{2}
& = \parallel \hspace{-1pt} \z^{t} - \z^{t+1}  + \z^{t+1} - \z^{t+2} + \cdots - \z^{\ast} \hspace{-1pt} \parallel_{2} \notag \\
%
& \leq \sum_{l = t}^{\infty} \parallel \hspace{-1pt} \z^{l} - \z^{l+1} \hspace{-1pt} \parallel_{2} \notag \\
&\leq \left( \sum_{j=1}^{\infty} \left( \eta_{1} \eta_{2} \right)^{j} \right) \parallel \hspace{-1pt} \z^{t+2} - \z^{t+1} \hspace{-1pt} \parallel_{2} \notag \\
&= \frac{\eta_{1} \eta_{2} }{1 - \eta_{1} \eta_{2} } \parallel \hspace{-1pt} \z^{t+2} - \z^{t+1} \hspace{-1pt} \parallel_{2}. 
\label{eq:tmprelation001}
\end{align}
Note that \eqref{eq:tmprelation001} is an upper bound of convergence rate. 

Similarly, we obtain
\begin{align}
\parallel \hspace{-1pt} \z^{t+1} \hspace{-1pt} - \hspace{-1pt} \z^{\ast} \hspace{-1pt} \parallel_{2}
&\leq \frac{1}{1 - \eta_{1} \eta_{2} } \parallel \hspace{-1pt} \z^{t+2} - \z^{t+1} \hspace{-1pt} \parallel_{2}.
\label{eq:tmprelation002}
\end{align}
From \eqref{eq:tmprelation001} and \eqref{eq:tmprelation002}, the following inequality is satisfied as 
\begin{align}
\parallel \hspace{-1pt} \z^{t+1} \hspace{-1pt} - \hspace{-1pt} \z^{\ast} \hspace{-1pt} \parallel_{2}
\leq \eta_{1} \eta_{2}
\parallel \hspace{-1pt} \z^{t} \hspace{-1pt} - \hspace{-1pt} \z^{\ast } \hspace{-1pt} \parallel_{2}. 
\end{align}
Thus, the convergence rate on Bregman Peaceman-Rachford splitting 
satisfies
\begin{align}
\parallel \hspace{-1pt} \z^{t} \hspace{-1pt} - \hspace{-1pt} \z^{\ast} \hspace{-1pt} \parallel_{2}
\leq \hspace{-1pt} 
(\eta_{1} \eta_{2})^{t}
\parallel \hspace{-1pt} \z^{0} \hspace{-1pt} - \hspace{-1pt} \z^{\ast} \hspace{-1pt} \parallel_{2}. 
\label{eq:crate_PR}
\end{align}

Next, we discuss the convergence rate of Bregman Douglas-Rachford splitting \eqref{eq:DRsplit2}. 
By using the triangle inequality, the contractive ratio of it is bounded by
\begin{align}
\parallel \hspace{-1pt} \z^{t+1} \hspace{-1pt} - \hspace{-1pt} \z^{\ast} \hspace{-1pt} \parallel_{2} 
&= \parallel \hspace{-1pt} \alpha C_{2} C_{1} (\z^{t}) + (1 \hspace{-1pt} - \hspace{-1pt} \alpha) \z^{t} \hspace{-1pt} - \hspace{-1pt} \z^{\ast} \hspace{-1pt} \parallel_{2} \notag \\
&\leq \alpha \hspace{-2pt} \parallel \hspace{-2pt} C_{2} C_{1} (\z^{t}) - \z^{\ast} \hspace{-2pt} \parallel_{2} 
+ (1 - \alpha) \hspace{-2pt} \parallel \hspace{-2pt} \z^{t} - \z^{\ast} \hspace{-2pt} \parallel_{2}
 \notag \\
&\leq \alpha \eta_{1} \eta_{2}
 \parallel \hspace{-2pt} \z^{t} - \z^{\ast} \hspace{-2pt} \parallel_{2}
+ (1 - \alpha) \hspace{-2pt} \parallel \hspace{-2pt} \z^{t} - \z^{\ast} \hspace{-2pt} \parallel_{2}
\notag \\
&= (1 \hspace{-1pt} - \hspace{-1pt} \alpha \hspace{-1pt} + \hspace{-1pt} \alpha 
\eta_{1} \eta_{2}) \parallel \hspace{-2pt} \z^{t} - \z^{\ast} \hspace{-2pt} \parallel_{2}.
\end{align}
Thus, the convergence rate of Bregman Douglas-Rachford splitting is bound by
\begin{align}
\hspace{-1pt} \parallel \hspace{-2pt} \z^{t} \hspace{-1pt} - \hspace{-1pt} \z^{\ast} \hspace{-2pt} \parallel_{2}
\leq \hspace{-1pt} 
(1 \hspace{-1pt} - \hspace{-1pt} \alpha \hspace{-1pt} + \hspace{-1pt} \alpha \eta_{1} \eta_{2})^{t}
 \hspace{-2pt}
\parallel \hspace{-2pt} \z^{0} \hspace{-1pt} - \hspace{-1pt} \z^{\ast} \hspace{-2pt} \parallel_{2}. 
\label{eq:crate_DR}
\end{align}

Bregman forward-backward splitting \eqref{eq:FBsplit} is composed of a $D$-forward
step for $G_{1}$ and a $D$-resolvent operator for $G_{2}$. The contractive ratio of these
operators is investigated in Theorems \ref{th:R1} and \ref{th:Q1}.
For input output pair $\w^{t} = F_{1} R_{2} (\w^{t-1}), \w^{t+1} = F_{1} R_{2} (\w^{t})$
the contraction is bound by
\begin{align}
\parallel \hspace{-1pt} \w^{t+1} \hspace{-1pt} - \hspace{-1pt} \w^{t} \hspace{-1pt} \parallel_{2} 
&\leq  \lambda
\parallel \hspace{-2pt} \w^{t} \hspace{-1pt} - \hspace{-1pt} \w^{t-1} \hspace{-2pt} \parallel_{2}, 
\end{align}
where $\lambda \geq 0$ is given by
\begin{align}
\lambda = 
\sqrt { \frac{ 1 - 2 \sigma_{\textrm{LB}, 1} + \sigma_{\textrm{UB}, 1}^{2} }{ (1 + \sigma_{\textrm{LB}, 2} )^{2} } }. 
\label{eq:lambda}
\end{align}
When the nonexpansive function condition for $\{ \sigma_{\textrm{UB}, 1}, \hspace{-1pt}
\sigma_{\textrm{LB}, 1} \}$, as shown in Fig. \ref{fig:nonexpansivecondition}, is
satisfied,  then application of Bregman forward-backward splitting generates a Cauchy-sequence as it satisfies $0 \leq \lambda \leq 1$. 
The convergence rate on Bregman forward-backward splitting is bound by
\begin{align}
\parallel \hspace{-2pt} \w^{t} - \w^{\ast} \hspace{-2pt} \parallel_{2} 
%
\leq \lambda^{t}
\parallel \hspace{-2pt} \w^{0} - \w^{\ast} \hspace{-2pt} \parallel_{2}.
\label{eq:crate_FB}
\end{align}

From the convergence rate predictions for Bregman Peaceman-Rachford splitting, Bregman Douglas
Rachford splitting and Bregman forward-backward splitting, given by \eqref{eq:crate_PR},
\eqref{eq:crate_DR} and \eqref{eq:crate_FB}, it is seen
that fast convergence is achieved when $\{ \sigma_{\mathrm{UB}, i}, \sigma_{\mathrm{LB}, i} \}$ approach 1 as in
\eqref{optimalsigmapair_DCayley} and
\eqref{optimalsigmapair_Dforward} because this implies $\eta_{i} = 0, \lambda = 0$.

\bibliographystyle{siamplain}

\end{document}